\documentclass[final,onefignum,onetabnum]{siamart250211}

\usepackage{lipsum}
\usepackage{amsfonts}
\usepackage{graphicx}
\usepackage{epstopdf}
\usepackage{algorithmic}

\ifpdf
  \DeclareGraphicsExtensions{.eps,.pdf,.png,.jpg}
\else
  \DeclareGraphicsExtensions{.eps}
\fi

\newsiamremark{remark}{Remark}
\newsiamremark{hypothesis}{Hypothesis}
\crefname{hypothesis}{Hypothesis}{Hypotheses}
\newsiamthm{claim}{Claim}

\headers{First-order methods on bounded-rank tensors}{B. Gao, R. Peng, and Y.-x. Yuan}

\title{First-order methods on bounded-rank tensors converging to stationary points\thanks{Submitted to the editors DATE.
\funding{This work was supported by the National Key R\&D Program of China (grant 2023YFA1009300). BG was supported by the Young Elite Scientist Sponsorship Program by CAST. YY was supported by the National Natural Science Foundation of China (grant No. 12288201).}}}

\author{Bin Gao\thanks{State Key Laboratory of Mathematical Sciences, Academy of Mathematics and Systems Science, Chinese  Academy of Sciences, China
		({\{gaobin,yyx\}@lsec.cc.ac.cn}).}
    \and Renfeng Peng\thanks{State Key Laboratory of Mathematical Sciences, Academy of Mathematics and Systems Science, Chinese  Academy of Sciences, and University of Chinese Academy of Sciences, China ({pengrenfeng@lsec.cc.ac.cn}).}
	\and Ya-xiang Yuan\footnotemark[2]
}

\usepackage{amsopn}

\usepackage{amssymb}

\newcommand{\vecx}{\mathbf{x}}
\newcommand{\vecr}{\mathbf{r}}

\newcommand{\vecy}{\mathbf{y}}

\newcommand{\mata}{\mathbf{A}}
\newcommand{\matb}{\mathbf{B}}

\newcommand{\matx}{\mathbf{X}}
\newcommand{\matu}{\mathbf{U}}
\newcommand{\matv}{\mathbf{V}}
\newcommand{\matW}{\mathbf{W}}

\newcommand{\matG}{\mathbf{G}}
\newcommand{\matI}{\mathbf{I}}

\newcommand{\matL}{\mathbf{L}}

\newcommand{\matR}{\mathbf{R}}

\newcommand{\matS}{\mathbf{S}}

\newcommand{\matU}{\mathbf{U}}

\newcommand{\tensA}{\mathcal{A}}

\newcommand{\tensC}{\mathcal{C}}

\newcommand{\tensG}{\mathcal{G}}

\newcommand{\tensM}{\mathcal{M}}

\newcommand{\tensV}{\mathcal{V}}
\newcommand{\tensW}{\mathcal{W}}
\newcommand{\tensY}{\mathcal{Y}}
\newcommand{\tensX}{\mathcal{X}}
\newcommand{\tensZ}{\mathcal{Z}}
\newcommand{\rank}{\mathrm{rank}}

\newcommand{\ranktc}{\mathrm{rank}_{\mathrm{tc}}}

\newcommand{\St}{\mathrm{St}}

\newcommand{\Span}{\mathrm{span}}

\newcommand{\subjectto}{\mathrm{s.\,t.}}
\newcommand{\T}{\mathsf{T}}

\newcommand{\frob}{\mathrm{F}}

\newcommand{\PGD}{$\mathrm{PGD}$}
\newcommand{\PtwoGD}{$\mathrm{P^2GD}$}

\newcommand{\RFD}{$\mathrm{RFD}$}

\usepackage{tikz}
\usepackage{tikz-3dplot}
\usetikzlibrary{patterns}
\usetikzlibrary{shapes,calc,shapes,arrows}
\usetikzlibrary {arrows.meta}

\makeatletter
\@addtoreset{equation}{section}
\makeatother

\DeclareMathOperator{\ten}{ten}
\DeclareMathOperator{\tangent}{T}
\DeclareMathOperator{\normal}{N}

\DeclareMathOperator{\retr}{R}

\DeclareMathOperator{\proj}{P}
\DeclareMathOperator{\approj}{\tilde{P}}

\DeclareMathOperator{\ho}{{HO}}
\DeclareMathOperator*{\argmin}{arg\,min}
\DeclareMathOperator*{\argmax}{arg\,max}

\begin{document}

\maketitle

\begin{abstract}
Provably finding stationary points on bounded-rank tensors turns out to be an open problem~[E. Levin, J. Kileel, and N. Boumal, \emph{Math. Program.}, 199 (2023), pp. 831--864] due to the inherent non-smoothness of the set of bounded-rank tensors. In contrast with bounded-rank matrices, tensors where some but not all modes are of full rank render essential difficulties in developing provable first-order methods. We resolve this problem by proposing two first-order methods with guaranteed convergence to stationary points. Specifically, we revisit the variational geometry of bounded-rank tensors and explicitly characterize its normal cones. Moreover, we propose gradient-related approximate projection methods that are provable to find stationary points, where the decisive ingredients are gradient-related vectors from tangent cones, line search along approximate projections, and rank-decreasing mechanisms near rank-deficient points. Numerical experiments on tensor completion validate that the proposed methods converge to stationary points across various rank parameters.
\end{abstract}



\section{Introduction}
We are concerned with the following low-rank tensor optimization problem based on Tucker decomposition~\cite{tucker1964extension},
\begin{equation}
    \begin{aligned}
        \min_{\tensX}\ &\ \ \ \ \ \ f(\tensX) \\
        \subjectto\ &\quad \tensX\in\tensM_{\leq\vecr}:=\{\tensX\in\mathbb{R}^{n_1\times n_2\times\cdots\times n_d}:\ranktc(\tensX)\leq\vecr\},
    \end{aligned}
    \label{eq: problem (P)}
\end{equation}
where $f:\mathbb{R}^{n_1\times n_2\times\cdots\times n_d}\to\mathbb{R}$ is a continuously differentiable function, the rank parameter $\vecr=(r_1,r_2,\dots,r_d)$ is an array of $d$ positive integers, and $\ranktc(\tensX)$ denotes the Tucker rank of $\tensX$. The feasible set $\tensM_{\leq\vecr}$ is referred to as the \emph{Tucker tensor variety}~\cite{gao2025low}. Tensor optimization based on Tucker decomposition appears to be prosperous in a broad range of applications~\cite{tucker1966some,vasilescu2003multilinear,koch2010dynamical,kressner2014low,kasai2016low,gao2025low,gao2025optimization}; see~\cite{kolda2009tensor,uschmajew2020geometric} for an overview.

We aim to develop methods that are guaranteed to find a stationary point $\tensX\in\tensM_{\leq\vecr}$ of~\cref{eq: problem (P)}, where the stationarity denotes that the projection of the anti-gradient $-\nabla f(\tensX)$ to the tangent cone $\tangent_\tensX\!\tensM_{\leq\vecr}$ at~$\tensX$ is equal to zero. However, due to the non-smoothness of $\tensM_{\leq\vecr}$, there exist sequences of points in $\tensM_{\leq\vecr}$ such that the projection of anti-gradient converges to zero, but stationarity does not hold at an accumulation point, i.e., the projection of anti-gradient can be non-zero; see an example in~\cite[Example 1]{gao2025low}. Such a phenomenon has been studied and named by \emph{apocalypse} in~\cite{levin2023finding}. In general, developing \emph{apocalypse-free} methods, where all accumulation points of sequences generated by such methods are stationary, turns out to be challenging, even in the matrix case~$d=2$~\cite[Proposition 2.10]{levin2023finding}. 

\paragraph{Related work and motivation}
We start with an overview of the existing methods on the set of bounded-rank matrices $\mathbb{R}_{\leq r}^{m\times n}:=\{\matx\in\mathbb{R}^{m\times n}:\rank(\matx)\leq r\}$. Jain et al.~\cite{jain2014iterative} proposed a projected gradient descent (\PGD) method, also known as the iterative hard thresholding method, in which an iterate is computed through $\matx^{(t+1)}=\proj_{\leq r}(\matx^{(t)}-s^{(t)}\nabla f(\matx^{(t)}))$ with stepsize $s^{(t)}$ and projection $\proj_{\leq r}$ onto $\mathbb{R}_{\leq r}^{m\times n}$. The \PGD~method is apocalypse-free~\cite{olikier2025projected}, but the iteration requires computing singular value decomposition for large matrices, resulting in prohibitively high computational cost. Based on the geometry of $\mathbb{R}_{\leq r}^{m\times n}$, Schneider and Uschmajew~\cite{schneider2015convergence} proposed another projected line search method (\PtwoGD) and a retraction-free variant (\RFD) that implement line search along the projected anti-gradient, which is able to save computational cost. Even though \PtwoGD\ and \RFD\ have accumulation points under typical assumptions---Armijo backtracking line search and compact sublevel set, an accumulation point is not necessary to be stationary if it is rank deficient, i.e., it has rank less than $r$. As a remedy, Olikier et al.~\cite{olikier2023apocalypse,olikier2026low} equipped \PtwoGD~and \RFD~methods with rank reduction, and developed two apocalypse-free methods. Furthermore, Olikier and Absil~\cite{olikier2023first} developed a framework for first-order optimization on general stratified sets of matrices. Note that one can also adopt Riemannian optimization methods~\cite{absil2009optimization,boumal2023intromanifolds} to minimize $f$ on the smooth manifold $\mathbb{R}_{r}^{m\times n}:=\{\matx\in\mathbb{R}^{m\times n}:\rank(\matx)=r\}$ (e.g.,~\cite{shalit2012online,vandereycken2013low}). The methods overlook the rank-deficient points $\mathbb{R}_{<r}^{m\times n}:=\mathbb{R}_{\leq r}^{m\times n}\setminus\mathbb{R}_{r}^{m\times n}$, and the classical convergence results in Riemannian optimization~\cite{boumal2019global} no longer hold if an accumulation point is rank deficient.

Another approach is to recast~\cref{eq: problem (P)} to a \emph{lifted} problem on a (product) manifold via a smooth parametrization, i.e., a parameter space $\tensM$ and a surjection $\phi$ such that $\phi(\tensM)=\mathbb{R}_{\leq r}^{m\times n}$. We refer to minimization of $f\circ\phi$ on $\tensM$ as the lifted problem. For instance, Levin et al.~\cite{levin2023finding} considered the product manifold $\mathbb{R}^{m\times r}\times\mathbb{R}^{n\times r}$ and the parametrization $(\matL,\matR)\mapsto\matL\matR^\top$. Rebjock and Boumal~\cite{rebjock2024optimization} considered a \emph{desingularization} approach to parametrize the matrix varieties. Since second-order stationary points of the lifted problem correspond to stationary points of~\cref{eq: problem (P)}, whereas first-order stationary points do not~\cite{ha2020equivalence,levin2024effect}, Levin et al.~\cite{levin2023finding} proposed a Riemannian trust-region method for the lifted problem.

Low-rank tensor optimization is significantly more complicated than low-rank matrix optimization~\cite{de2008tensor,hillar2013most}. Nevertheless, similar to the matrix case, low-rank tensor optimization mainly has three different geometric approaches: optimization on manifolds, optimization on varieties, and optimization via parametrizations. Since $\tensM_{\vecr}:=\{\tensX\in\mathbb{R}^{n_1\times n_2\times\cdots\times n_d}:\ranktc(\tensX)=\vecr\}$ is a smooth manifold~\cite{uschmajew2013geometry}, Kressner et al.~\cite{kressner2014low} proposed a Riemannian conjugate gradient method. The convergence results in Riemannian optimization~\cite{boumal2019global} are still not applicable since~$\tensM_{\vecr}$ is not closed. Gao et al.~\cite{gao2025low} delved into the geometry of its closure~$\tensM_{\leq\vecr}$, and proposed line-search methods and a  rank-adaptive method. For optimization via a smooth parametrization, Kasai and Mishra~\cite{kasai2016low} considered a parametrization of $\tensM_{\vecr}$ via a quotient manifold. Desingularization approaches for bounded-rank tensors was proposed~\cite{gao2024desingularization}. We refer to~\cite{steinlechner2016riemannian,dong2022new,gao2024riemannian} for tensor optimization in other tensor formats.

For bounded-rank tensors, low-rank tensor optimization still suffers from apocalypse at rank-deficient accumulation points $\tensM_{\leq\vecr}\setminus\tensM_{\vecr}$; see~\cite[Example 1]{gao2025low}. Unfortunately, the aforementioned methods for~\cref{eq: problem (P)} are not apocalypse-free. In fact, provably finding stationary points on the set of bounded-rank tensors remains an open problem~\cite[Example 3.14]{levin2023finding}. One may consider generalizing existing apocalypse-free methods for matrices to tensors. However, the intricate geometry of low-rank tensors hinders generalization, making it much more elusive to circumvent apocalypse than matrices. Specifically, it experiences the following pains.
\begin{enumerate}
    \item Since the Tucker rank is an array rather than an integer, there exist points in $\tensM_{\leq\vecr}\setminus(\tensM_{\vecr}\cup\tensM_{<\vecr})$ such that some modes of a tensor are of full rank and the others are rank deficient. As a consequence, Tucker tensor varieties fail to satisfy the general assumptions in the framework~\cite{olikier2023first} tailored for stratified sets, and the developed theoretical results can not be simply generalized to tensors. In addition, practical optimality conditions for such points remain unknown. 
    \item In contrast with the set of bounded-rank matrices, the metric projection onto the tangent cone of $\tensM_{\leq\vecr}$ does not enjoy a closed-form expression, necessitating approximate projections. 
    \item Optimization methods via parametrizations of $\tensM_{\leq\vecr}$ no longer guarantee to find stationary points~\cite{levin2024effect,gao2024desingularization}. For instance, one may reformulate~\cref{eq: problem (P)} as a lifted problem by directly optimizing over the factors in Tucker format, in a same fashion as matrix case. However, the parametrization introduces non-uniqueness and degeneracies, and second-order stationary points of the lifted problem do not necessarily correspond to stationary points of~\cref{eq: problem (P)}, as shown in~\cite[\S 4.3]{gao2024desingularization}. As a result, the desirable properties of parametrizations of low-rank matrices no longer hold for tensors. 
\end{enumerate}

\paragraph{Contribution}
We propose first-order methods on Tucker tensor varieties that are able to accumulate at stationary points. To this end, two components are central to our methods: appropriate search directions and rank-decreasing mechanisms. 

We revisit the geometry of Tucker tensor varieties and develop an explicit representation of the normal cone of $\tensM_{\leq\vecr}$. Remarkably, the normal cone is a linear space, which allows a practical validation of optimality of~\cref{eq: problem (P)} by projecting the anti-gradient to the normal cone with a closed-form expression. By contrast, projection onto the tangent cone is computationally intractable. We develop an approximate projection onto the tangent cone constructed via singular value decomposition. The proposed approximate projection is proved to satisfy the \emph{angle condition}, paving the way for first-order methods with guaranteed convergence. 

Subsequently, we propose a gradient-related approximate projection method with rank decrease (GRAP-R), which mainly consists of two steps: rank-decreasing mechanism and line search; see a general framework in~\cref{fig: general illustration}. Specifically, given a point $\tensX^{(t)}\in\tensM_{\leq\vecr}$, the mechanism monitors the singular values of unfolding matrices of $\tensX^{(t)}$ along each mode, and truncates $\tensX^{(t)}$ to lower-rank candidates if rank deficiency is detected. Moreover, we perform retraction-based line searches, which preserve feasibility, along approximate projections at every lower-rank candidate. The next iterate $\tensX^{(t+1)}$ is selected from all candidates that attains the lowest function value after line search.

\begin{figure}[htbp]
    \centering
    \includegraphics[width=\linewidth]{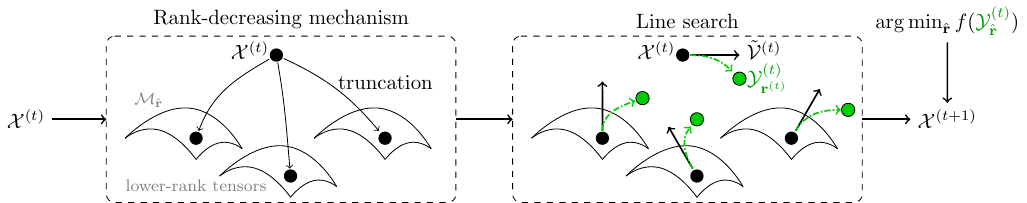}
    \caption{A general framework of proposed first-order methods. The surfaces represent sets of fixed-rank tensors, and $\tilde{\tensV}^{(t)}$ is a search direction. Green points are generated from lower-rank candidates via line search.}
    \label{fig: general illustration}
\end{figure}

By integrating the rank-decreasing mechanism with line search, we prove that the GRAP-R method converges to a stationary point through a \emph{sufficient decrease condition} defined in a neighborhood of a non-stationary point; see a roadmap that achieves the condition in~\cref{fig: proof sketch}. Specifically, the rank-decreasing mechanism enables the GRAP-R method to consider the nearby rank-deficient points where apocalypse often occurs. Line search along approximate projections contributes to an Armijo condition. Both steps protect the sequences generated by GRAP-R from converging to non-stationary points, which renders GRAP-R apocalypse-free.

\begin{figure}[htbp]
    \centering
    \includegraphics[width=\textwidth]{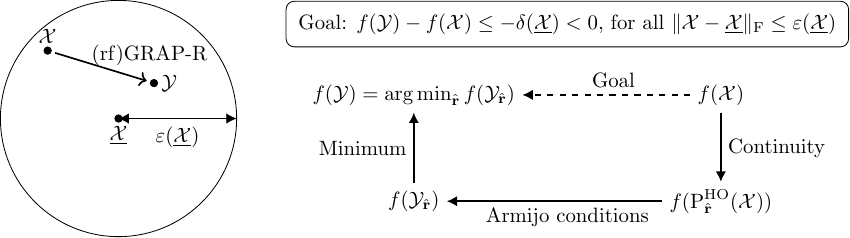}
    \caption{Definition and validation of the sufficient decrease condition in convergence analysis. In a neighborhood of a non-stationary point $\underline{\tensX}$, the function value decreases by more than a positive constant $\delta(\underline{\tensX})$ from $\tensX$ to $\tensY$. The roadmap illustrates how the (rf)GRAP-R method achieves the sufficient decrease condition.}
    \label{fig: proof sketch}
\end{figure}

Alternatively, we consider retraction-free search directions whose line searches never leave $\tensM_{\leq\vecr}$. In a similar manner, we adopt the framework in~\cref{fig: general illustration} and propose a new retraction-free GRAP-R method (rfGRAP-R). In contrast with the GRAP-R method where all singular values of unfolding matrices are required, the rfGRAP-R method only detects the smallest singular value of unfolding matrices, which is in favor of saving computational costs. We prove that the rfGRAP-R method also converges to stationary points by verifying a sufficient decrease condition in~\cref{fig: proof sketch}. 

In addition, numerical results validate that the proposed methods indeed converge to stationary points. In practice, if the prescribed rank is exactly the true rank of a dataset, the GRAP-R method performs better than the rfGRAP-R method. If the rank is over-estimated or the rank parameter becomes large, the rfGRAP-R method performs better than GRAP-R, since the retraction-free property is able to reduce the computational cost by avoiding exhaustive exploration of multiple lower-rank candidates.

\paragraph{Organization}
We introduce the geometry of Tucker tensor varieties and notation for tensor operations in~\cref{sec: Preliminaries}. In~\cref{sec: normal cone and approj}, we provide an explicit representation of the normal cone of Tucker tensor varieties, and propose approximate projections. Based on two different projections, we propose two first-order methods and prove that all accumulation points of sequences generated by these methods are stationary in~\cref{sec: GRAP-R,sec: rfGRAP-R}. Section~\ref{sec: experiments} reports the numerical performance of proposed methods in tensor completion. Finally, we draw the conclusion in~\cref{sec: conclusion}.

\section{Geometry of Tucker tensor varieties}\label{sec: Preliminaries}
In this section, we introduce the preliminaries of variational analysis for low-rank matrices. Then, we describe the notation in tensor operations and Tucker tensor varieties.

Given a nonempty subset $C$ of $\mathbb{R}^{n_1\times n_2\times\cdots\times n_d}$, the (Bouligand) \emph{tangent cone} of~$C$ at $\tensX\in C$ is defined by $\tangent_\tensX\!C:=\{\tensV\in\mathbb{R}^{n_1\times n_2\times\cdots\times n_d}:\exists t^{(i)}\to 0,\ \tensX^{(i)}\to \tensX\text{ in }C,\ \subjectto\ (\tensX^{(i)}-\tensX)/t^{(i)}\to\tensV\}$. The set $\normal_\tensX\! C=(\tangent_\tensX\! C)^\circ:=\{\tensW\in\mathbb{R}^{n_1\times n_2\times\cdots\times n_d}:\langle\tensW,\tensV\rangle\leq 0\ \text{for all}\ \tensV\in\tangent_\tensX\! C\}$ is called the \emph{normal cone} of $C$ at $\tensX$. Note that if~$C$ is a manifold, the tangent cone $\tangent_\tensX\!C$ (normal cone $\normal_\tensX\!C$) is a linear space and referred to as the tangent space (normal space). Given a continuously differentiable function $f$ on~$C$, a point $x\in C$ is stationary of $f$ if $\|\proj_{\tangent_x\!C}(-\nabla f(x))\|_\frob=0$, or $-\nabla f(x)\in\normal_x\!C$ equivalently.

\subsection{Notation for matrix computations}\label{subsec: low-rank matrix}
Let $m,n,r$ be positive integers satisfying $r\leq\min\{m,n\}$. Given a matrix $\matx\in\mathbb{R}^{m\times n}$, the image of $\matx$ and its orthogonal complement are defined by $\Span(\matx):=\{\matx\vecy:\vecy\in\mathbb{R}^n\}\subseteq\mathbb{R}^m$ and $\Span(\matx)^\perp:=\{\vecy\in\mathbb{R}^m:\langle\vecx,\vecy\rangle=0\ \text{for all}\ \vecx\in\Span(\matx)\}$ respectively. The set $\St(r,n):=\{\matx\in\mathbb{R}^{n\times r}:\matx^\top\matx=\matI_r\}$ is the \emph{Stiefel manifold}. Given $\underline{\matx}\in\mathbb{R}^{m\times n}$ and $\Delta>0$, we denote the closed ball $B[\underline{\matx},\Delta]:=\{\matx\in\mathbb{R}^{m\times n},\|\matx-\underline{\matx}\|_\frob\leq\Delta\}$. The singular values of $\matx\in\mathbb{R}^{m\times n}$ is denoted by $\sigma_1\geq\sigma_2\geq\cdots\geq\sigma_{\min\{m,n\}}$. The smallest non-zero singular value of $\matx$ is denoted by $\sigma_{\min}(\matx)=\sigma_{\rank(\matx)}(\matx)$. The $\Delta$-rank of $\matx$ is defined by $\rank_\Delta\matx=\min\{i:\sigma_{i+1}\leq\Delta\}$. 
\begin{proposition}[{\cite[Proposition 2.5]{olikier2026low}}]\label{prop: rank delta}
    Given $\underline{\matx}\in\mathbb{R}^{m\times n}_{\underline{r}}$ and $\Delta>0$. It holds that $\rank_\Delta\matx\leq\underline{r}$ for all $\matx\in B[\underline{\matx},\Delta]$. If $\underline{\matx}\neq 0$, it holds that $\underline{r}\leq \rank(\matx)$ for all $\matx\in\mathbb{R}^{m\times n}$ and $\|\matx-\underline{\matx}\|_\frob<\sigma_{\min}(\underline{\matx})$.
\end{proposition}

\subsection{Geometry of tensors in Tucker decomposition}
We introduce notation for tensor operations. Denote the index set $\{1,2,\dots,n\}$ by~$[n]$. The inner product between two tensors $\tensX,\tensY\in\mathbb{R}^{n_1\times n_2\times\cdots\times n_d}$ is defined by $\langle\tensX,\tensY\rangle := \sum_{i_1=1}^{n_1} \cdots \sum_{i_d=1}^{n_d} \tensX({i_1,\dots,i_d})\tensY({i_1,\dots,i_d})$. The Frobenius norm of a tensor $\tensX$ is defined by $\|\tensX\|_\mathrm{F}:=\sqrt{\langle\tensX,\tensX\rangle}$. The mode-$k$ unfolding of a tensor $\tensX \in \mathbb{R}^{n_1 \times\cdots\times n_d}$ is denoted by a matrix $\matx_{(k)}\in\mathbb{R}^{n_k\times n_{-k}} $ for $k=1,\dots,d$, where $n_{-k}:=\prod_{i\neq k}n_i$. The $ (i_1,i_2,\dots,i_d)$-th entry of $\tensX$ corresponds to the $(i_k,j)$-th entry of $\matx_{(k)}$, where
$ j = 1 + \sum_{\ell \neq k, \ell = 1}^d(i_\ell-1)J_\ell$ with $J_\ell = \prod_{m = 1,m \neq k}^{\ell-1} n_m$. The tensorization operator maps a matrix $\matx_k\in\mathbb{R}^{n_k\times n_{-k}}$ to a tensor $\ten_{(k)}(\matx_k)\in\mathbb{R}^{n_1\times\cdots\times n_d}$ defined by $\ten_{(k)}(\matx_k)(i_1,\dots,i_d)=\matx_k(i_k,1 + \sum_{\ell \neq k, \ell = 1}^d(i_\ell-1)J_\ell)$ for $(i_1,\dots,i_d)\in[n_1]\times\cdots\times[n_d]$. Note that $\ten_{(k)}(\matx_{(k)})=\tensX$ holds for fixed $n_1,\dots,n_d$. Therefore, the tensorization operator is invertible. The $k$-mode product of a tensor $\tensX$ and a matrix $\mata\in\mathbb{R}^{M\times n_k}$ is denoted by $\tensX\times_k\mata\in\mathbb{R}^{n_1\times\cdots\times M\times\cdots\times n_d}$, where the $ (i_1,\dots,i_{k-1},j,i_{k+1},\dots,i_d)$-th entry of $\tensX\times_k\mata$ is $\sum_{i_k=1}^{n_k}x_{i_1\dots i_d}a_{ji_k}$. It holds that $(\tensX\times_k\mata)_{(k)}=\mata\matx_{(k)}$. The Kronecker product of two matrices $\mata\in\mathbb{R}^{m_1\times n_1}$ and $\matb\in\mathbb{R}^{m_2\times n_2}$ is an $(m_1m_2)$-by-$(n_1n_2)$ matrix defined by $\mata\otimes\matb:=(a_{ij}\matb)_{ij}$. Given two vectors $\vecx,\vecy\in\mathbb{R}^d$, we denote $\vecx\leq\vecy$ ($\vecx<\vecy$) if $x_i\leq y_i$ ($x_i<y_i$) for all $i\in[d]$.

\begin{definition}[Tucker decomposition]
    Given a tensor $\tensX \in \mathbb{R}^{n_1 \times n_2\times\cdots\times n_d}$, 
    the Tucker decomposition is \[\tensX =\tensG\times_1\matu_1\times_2\matu_2\cdots\times_d\matu_d=\tensG\times_{k=1}^d\matu_k,\]
    where $\tensG\in\mathbb{R}^{r_1 \times r_2\times\cdots\times r_d}$ is a core tensor, $\matu_k\in\St(r_k,n_k)$ are factor matrices with orthogonal column and $r_k=\rank(\matx_{(k)})$ for all $k\in[d]$. 
\end{definition}

The Tucker rank of a tensor $\tensX$ is defined by $$\ranktc(\tensX):=(\rank(\matx_{(1)}),\rank(\matx_{(2)}),\dots,\rank(\matx_{(d)})).$$ 
Figure~\ref{fig: 3D Tucker} depicts the Tucker decomposition of a third-order tensor. For a $d$-th order tensor $\tensA$, it holds that $\tensA\in\otimes_{k=1}^d\Span(\matu_k)$ if and only if there exists $\tensC\in\mathbb{R}^{r_1\times r_2\times\cdots\times r_d}$ such that $\tensA=\tensC\times_{k=1}^d\matu_k$. Note that the mode-$k$ unfolding of $\tensX=\tensG\times_{k=1}^d\matu_k$ satisfies 
\[\matx_{(k)}=\matu_k\matG_{(k)}\left(\matu_d\otimes\cdots\otimes\matu_{k+1}\otimes\matu_{k-1}\otimes\cdots\otimes\matu_{1}\right)^\top=\matu_k\matG_{(k)}((\matu_j)^{\otimes j\neq k})^\top.\]

\begin{figure}[htbp]
    \centering
    \includegraphics[scale=0.833]{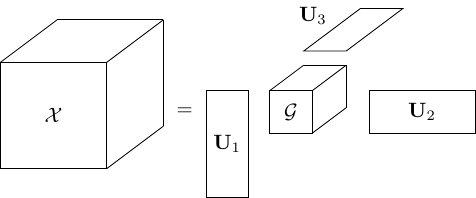}
    \caption{Tucker decomposition of a third-order tensor.}
    \label{fig: 3D Tucker}
\end{figure}

\paragraph{Projection onto $\tensM_{\leq\vecr}$}
Given a tensor $\tensA\in\mathbb{R}^{n_1\times n_2\times\cdots\times n_d}$, the metric projection of~$\tensA$ onto Tucker tensor variety $\tensM_{\leq\vecr}$ is defined by
\begin{equation*}
    \proj_{\leq\vecr}(\tensA):=\argmin_{\tensX\in\tensM_{\leq\vecr}}\|\tensA-\tensX\|_\frob^2.
\end{equation*}
In contrast with the matrix case where the metric projection can be computed by singular value decomposition, $\proj_{\leq\vecr}(\tensA)$ does not have a closed-form expression in general~\cite{de2000multilinear}. Nevertheless, one can apply \emph{higher-order singular value decomposition} (HOSVD) to yield a quasi-optimal solution. Specifically, the HOSVD procedure sequentially applies the best rank-$r_k$ approximation operator $\proj^k_{\leq r_k}$ to each mode of $\tensA$ for $k\in[d]$, i.e,
\begin{equation}
    \label{eq: HOSVD}
    \proj_{\leq\vecr}^\mathrm{HO}(\tensA):=\proj_{\leq r_d}^d(\proj_{\leq r_{d-1}}^{d-1}\cdots(\proj_{\leq r_1}^1(\tensA))),
\end{equation}
where $\proj_{\leq r_k}^k(\tensA)=\ten_{(k)}(\proj_{\leq r_k}(\mata_{(k)}))$. The quasi-optimality~\cite[Lemma 2.6]{grasedyck2010hierarchical}
\begin{equation}
    \label{eq: quasi-optimal HOSVD}
    \|\tensA-\proj_{\leq\vecr}^{\ho}(\tensA)\|_\frob\leq\sqrt{d}\|\tensA-\proj_{\leq\vecr}(\tensA)\|_\frob
\end{equation}
enables HOSVD to be an approximate projection onto $\tensM_{\leq\vecr}$.

A mapping $\retr:\bigcup_{\tensX\in \tensM_{\leq\vecr}}\{\tensX\}\times\tangent_\tensX\!\tensM_{\leq\vecr}\to \tensM_{\leq\vecr}$ is called a \emph{retraction}~\cite[\S 3.1.2]{hosseini2019gradient} if for all $\tensX\in\tensM_{\leq\vecr}$ and $\tensV\in\tangent_\tensX\!\tensM_{\leq\vecr}$, it holds that $\lim_{t\to 0^+}(\retr_\tensX(t\tensV)-\tensX-t\tensV)/t=0$. It follows from~\cite[Proposition 2]{gao2025low} that $\proj_{\leq\vecr}^{\ho}$ provides a retraction on $\tensM_{\leq\vecr}$.

\paragraph{Tucker tensor varieties}
Given a tuple $\vecr=(r_1,r_2,\dots,r_d)$, the set of tensors of bounded Tucker rank $\tensM_{\leq\vecr}=\{\tensX\in\mathbb{R}^{n_1\times n_2\times \cdots\times n_d}:\ranktc(\tensX)\leq\vecr\}$ is referred to as the Tucker tensor variety. Given a tensor $\tensX=\tensG\times_{k=1}^d\matu_k$ with $\ranktc(\tensX)=\underline{\vecr}\leq\vecr$, a~vector $\tensV$ in the tangent cone $\tangent_\tensX\!\tensM_{\leq\vecr}$ has an explicit parametrization~\cite[Theorem~1]{gao2025low}
\begin{equation}
    \label{eq: Tucker tangent cone}
        \tensV=\tensC\times_{k=1}^d\begin{bmatrix}
            \matu_k & \matu_{k,1}
        \end{bmatrix}+\sum_{k=1}^d\tensG\times_k\dot{\matu}_k\times_{j\neq k}\matu_j,
\end{equation}
where $\tensC\in\mathbb{R}^{r_1\times r_2\times\cdots\times r_d}$, $\matu_{k,1}\in\St(r_k-\underline{r}_k,n_k)$, and $\dot{\matu}_k\in\mathbb{R}^{n_k\times\underline{r}_k}$ are arbitrary that satisfy $\matu_k^\top[\matu_{k,1}^{}\ \dot{\matu}_k^{}]=0$ and $\matu_{k,1}^\top\dot{\matu}_k^{}=0$ for $k\in[d]$.

\section{Normal cones and approximate projections}\label{sec: normal cone and approj}
In this section, we develop an explicit representation of the normal cone, providing practical optimality conditions of~\cref{eq: problem (P)}. To facilitate provable first-order methods, we also develop approximate projections satisfying the angle conditions.

\subsection{Normal cones and optimality conditions}
Given $\tensX=\tensG\times_{k=1}^d\matu_k\in\tensM_{\leq\vecr}$ with $\underline{\vecr}=\ranktc(\tensX)$ and $\tangent_\tensX\!\tensM_{\leq\vecr}$ in~\cref{eq: Tucker tangent cone}, an element $\tensW\in\mathbb{R}^{n_1\times n_2\times \cdots\times n_d}$ in the normal cone $\normal_\tensX\!\tensM_{\leq\vecr}$ satisfies that
\[\langle\tensW,\tensV\rangle=\langle\tensW,\tensC\times_{k=1}^d\begin{bmatrix}
        \matu_k & \matu_{k,1}
    \end{bmatrix}+\sum_{k=1}^d\tensG\times_k\dot{\matu}_k\times_{j\neq k}\matu_j\rangle\leq 0\]
for arbitrary $\tensC\in\mathbb{R}^{r_1\times r_2\times\cdots\times r_d}$, $\matu_{k,1}\in\St(r_k-\underline{r}_k,n_k)$, and $\dot{\matu}_k\in\mathbb{R}^{n_k\times\underline{r}_k}$ satisfying $\matu_k^\top[\matu_{k,1}^{}\ \dot{\matu}_k^{}]=0$ and $\matu_{k,1}^\top\dot{\matu}_k^{}=0$ for $k\in[d]$. Since $-\tensV=(-\tensC)\times_{k=1}^d\begin{bmatrix}
    \matu_k & \matu_{k,1}
\end{bmatrix}+\sum_{k=1}^d\tensG\times_k(-\dot{\matu}_k)\times_{j\neq k}\matu_j\in\tangent_\tensX\!\tensM_{\leq\vecr}$, we have $\langle\tensW,\tensV\rangle=0$. By letting one factor free while setting all the remaining factors to zero in $(\tensC,\dot{\matu}_{1}, \dot{\matu}_{2}, \dots, \dot{\matu}_{d})$, we obtain that
\[\langle\tensW,\tensC\times_{k=1}^d\begin{bmatrix}
        \matu_k & \matu_{k,1}
        \end{bmatrix}\rangle=0\quad\text{and}\quad\langle\tensW,\tensG\times_k\dot{\matu}_k\times_{j\neq k}\matu_j\rangle=0\]
for all $k\in[d]$.

Denote the set of rank-deficient index \[I=\{k\in[d]:\underline{r}_k<r_k\}.\] 
We observe that $\tensC\times_{k\in I}\begin{bmatrix}
    \matu_k & \matu_{k,1}
\end{bmatrix}\times_{k\in[d]\setminus I}\matu_k$ can represent any rank-$1$ tensor in $S_1\otimes S_2\otimes\cdots\otimes S_d$ with $S_k=\Span(\matu_{k})$ if $r_k=\underline{r}_k$ or $S_k=\mathbb{R}^{n_{k}}$ otherwise. Therefore, $\langle\tensW,\tensC\times_{k=1}^d\begin{bmatrix}
        \matu_k & \matu_{k,1}
        \end{bmatrix}\rangle=0$ 
is equivalent to $\tensW\in(S_1\otimes S_2\otimes\cdots\otimes S_d)^\perp$. Additionally, the tensor~$\tensW$ satisfies $\langle\tensW,\tensG\times_k\dot{\matu}_k\times_{j\neq k}\matu_j\rangle=0$ for all $\dot{\matu}_k$, i.e., $\proj_{\matu_k}^\perp\!\matW_{(k)}\matv_k\matG_{(k)}^\top=0$ for all $k\in[d]$, where the projection matrix $\proj_{\matu_k}^\perp=\matI_{n_k}-\matu_k^{}\matu_k^\top$ and $\matv_k=\matu_j^{\otimes j\neq k}$. If~$\underline{r}_k<r_k$, it follows from $\Span(\matu_k)^\perp\subseteq\mathbb{R}^{n_k}=S_k$ that $\otimes_{j=1}^{k-1}\Span(\matu_j)\otimes\Span(\matu_k)^\perp\otimes_{j=k+1}^{d}\Span(\matu_j)\subseteq S_1\otimes S_2\otimes\cdots\otimes S_d$, i.e., $\tensW\in(S_1\otimes S_2\otimes\cdots\otimes S_d)^\perp$ indicates $\proj_{\matu_k}^\perp\!\matW_{(k)}\matv_k\matG_{(k)}^\top=0$ for~$k\in I$. Consequently, in the light of the decomposition $\mathbb{R}^{n_1\times n_2\times\cdots\times n_d}=\otimes_{k=1}^d(\Span(\matu_k)+\Span(\matu_k)^\perp)$, the normal cone can be represented by the following proposition.
\begin{proposition}[normal cone]
    Given $\tensX=\tensG\times_{k=1}^d\matu_k\in\tensM_{\leq\vecr}$ with $\underline{\vecr}:=\ranktc(\tensX)$, we denote the index set $I=\{k\in[d]:\underline{r}_k<r_k\}$. The normal cone of $\tensM_{\leq\vecr}$ at $\tensX$ can be represented by
    \[\normal_\tensX\!\tensM_{\leq\vecr}=(S_1\otimes S_2\otimes\cdots\otimes S_d)^\perp\cap\{\tensW\in\mathbb{R}^{n_1\times\cdots\times n_d}:\proj_{\matu_k}^\perp\!\matW_{(k)}\matv_k\matG_{(k)}^\top=0,\ k\in[d]\setminus I\},    
    \]
    where $S_k=\Span(\matu_{k})$ if $k\in[d]\setminus I$ or $S_k=\mathbb{R}^{n_{k}}$ otherwise. More precisely, an element $\tensW\in\normal_\tensX\!\tensM_{\leq\vecr}$ can be explicitly parametrized by
    \[\tensW=\sum_{i_1,i_2,\dots,i_d=0}^1(\sum_{k\in[d]\setminus I}i_k^2)\tensC_{i_1,\dots,i_d}\times_{k=1}^d\matu_k^{(i_k)},\]
    where $\tensC_{i_1,\dots,i_d}\in\mathbb{R}^{\tilde{n}_1\times\tilde{n}_2\times\cdots\times\tilde{n}_d}$ is arbitrary satisfying $(\tensC_{i_1,\dots,i_d})_{(k)}\matG_{(k)}^\top=0$ if $i_k=1$ and $i_j=0$ for $j\in[d]\setminus\{k\}$, $\tilde{n}_k=(1-i_k)\underline{r}_k+i_k(n_k-\underline{r}_k)$; $\matu_k^{(0)}:=\matu_k$ and $\matu_k^{(1)}:=\matu_k^\perp\in\St(n_k-\underline{r}_k,n_k)$ and $\Span(\matu_k^\perp)=\Span(\matu_k)^\perp$.
\end{proposition}

It is worth noting that: 1) The normal cone $\normal_\tensX\!\tensM_{\leq\vecr}$ is indeed a linear space for all $\tensX\in\tensM_{\leq\vecr}$ since it is the intersection of two linear spaces; 2) $(\sum_{k\in[d]\setminus I}i_k^2)$ is an indicator that only selects components contributing to the normal space, which excludes tensors in $S_1\otimes S_2\otimes\cdots\otimes S_d$; 3) $\tensC_{i_1,\dots,i_d}\times_{k=1}^d\matu_k^{(i_k)}$ is a tensor in the tensor product space indexed by $(i_1,\dots,i_d)$. Specifically, recall that $\mathbb{R}^{n_1\times n_2\times\cdots\times n_d}=\otimes_{k=1}^d(\Span(\matu_k)+\Span(\matu_k)^\perp)$, we decompose the space into $2^d$ orthogonal spaces indexed by $(i_1,\dots,i_d)\in\{0,1\}^d$, where $i_k=0$ implies $\Span(\matU_k)$ and $i_k=1$ implies $\Span(\matU_k)^\perp$.~\footnote{For instance, the indices $(0,0,\dots,0)$ and $(1,0,\dots,0)$ correspond to the tensor product space $\otimes_{k=1}^d\Span(\matu_k)$ and $\Span(\matu_1)^\perp\otimes\Span(\matu_2)\otimes\cdots\otimes\cdots\Span(\matu_d)$.}
In particular, an element $\tensW\in\normal_\tensX\!\tensM_{\leq\vecr}$ for $d=3$ can be explicitly parametrized by elements in $(\Span(\matu_1)+\Span(\matu_1^\perp))\otimes(\Span(\matu_2)+\Span(\matu_2^\perp))\otimes(\Span(\matu_3)+\Span(\matu_3^\perp))$:
\[
    \tensW=\left\{
        \begin{array}{ll}
            0 & \text{if }I=\{1,2,3\}, \\
            \sum_{i_1,i_2=0}^1\tensC_{i_1,i_2,1}\times_3\matu_3^\perp\times_{k\in I}((1-i_k)\matu_k+i_k\matu_k^\perp) & \text{if }I=\{1,2\}, \\
            \sum_{i_2,i_3=0}^1\tensC_{1,i_2,i_3}\times_1\matu_1^\perp\times_{k\in I}((1-i_k)\matu_k+i_k\matu_k^\perp) & \text{if }I=\{2,3\}, \\
            \sum_{i_1,i_3=0}^1\tensC_{i_1,1,i_3}\times_2\matu_2^\perp\times_{k\in I}((1-i_k)\matu_k+i_k\matu_k^\perp) & \text{if }I=\{1,3\}, \\
            \sum_{i_1,i_2,i_3=0}^1(i_2^2+i_3^2)\tensC_{i_1,i_2,i_3}\times_{k=1}^d((1-i_k)\matu_k+i_k\matu_k^\perp) & \text{if }I=\{1\}, \\
            \sum_{i_1,i_2,i_3=0}^1(i_1^2+i_3^2)\tensC_{i_1,i_2,i_3}\times_{k=1}^d((1-i_k)\matu_k+i_k\matu_k^\perp) & \text{if }I=\{2\}, \\
            \sum_{i_1,i_2,i_3=0}^1(i_1^2+i_2^2)\tensC_{i_1,i_2,i_3}\times_{k=1}^d((1-i_k)\matu_k+i_k\matu_k^\perp) & \text{if }I=\{3\}, \\
            \sum_{i_1,i_2,i_3=0}^1(i_1^2+i_2^2+i_3^2)\tensC_{i_1,i_2,i_3}\times_{k=1}^d((1-i_k)\matu_k+i_k\matu_k^\perp) & \text{if }I=\emptyset;
        \end{array}
    \right.
\]
see an illustration in~\cref{fig: normal cones}. The normal cone $\normal_\tensX\!\tensM_{\leq\vecr}$ boils down to the normal space $\normal_\tensX\!\tensM_{\vecr}$ if $\underline{\vecr}=\vecr$ and to $\{0\}$ if $\underline{\vecr}<\vecr$.

\begin{figure}[htbp]
    \centering
    \includegraphics[width=0.32\textwidth]{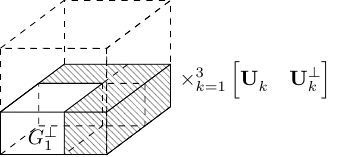}
    \includegraphics[width=0.32\textwidth]{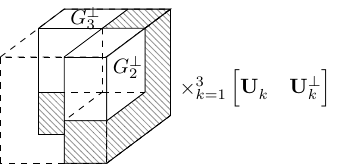}
    \includegraphics[width=0.32\textwidth]{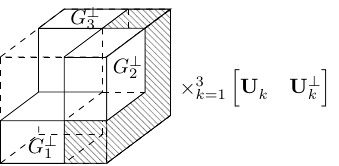}
    \caption{Illustration of a vector in normal cone for $d=3$. $G_k^\perp=\{\tilde{\tensG}:\tilde{\matG}_{(k)}\in\mathbb{R}^{(n_k-\underline{r}_k)\times r_{-k}},\tilde{\matG}_{(k)}^{}\matG_{(k)}^\top=0\}$ for $k\in[3]$. Left: $I=\{2,3\}$. Middle: $I=\{1\}$. Right: $I=\emptyset$.}
    \label{fig: normal cones}
\end{figure}

We observe that $\normal_\tensX\!\tensM_{\leq\vecr}$ is a linear space for all $\tensX\in\tensM_{\leq\vecr}$, which facilitates practical approaches to verify the stationary points of a continuously differentiable function~$f$. Specifically, a point $\tensX^*\in\tensM_{\leq\vecr}$ with $\ranktc(\tensX^*)=\underline{\vecr}^*$ and Tucker decomposition $\tensX^*=\tensG^*\times_{k=1}^d\matu_k^*$ is stationary of $f$ if $-\nabla f(\tensX^*)\in\normal_{\tensX^*}\!\tensM_{\leq\vecr}$, or 
\begin{equation*}
    \nabla f(\tensX^*)\times_{k=1}^d\proj_{S_k^*}=0\ \text{and}\ 
    (\nabla f(\tensX^*)\times_k^{}\proj_{\matu_k^*}^\perp\times_{j\neq k}^{}(\matu_j^*)^\top)_{(k)}^{}
    (\matG_{(k)}^{*})^\top = 0\ \text{for}\ k\in[d]\setminus I
\end{equation*} 
equivalently. For $\underline{\vecr}^*=\vecr$ and $\underline{\vecr}^*<\vecr$, the result boils down to $\proj_{\tangent_{\tensX^*}\!\tensM_{\vecr}}(\nabla f(\tensX^*))=0$ and $\nabla f(\tensX^*)=0$ respectively, which coincides with~\cite[Proposition 3]{gao2025low}. In summary, the optimality condition based on $\normal_\tensX\!\tensM_{\leq\vecr}$ can be verified at every point in $\tensM_{\leq\vecr}$, including points in~$\tensM_{\leq\vecr}\setminus(\tensM_{\vecr}\cup\tensM_{<\vecr})$.

\subsection{Approximate projections}
We consider appropriate search directions via approximate projections. Given a tensor $\tensA\in\mathbb{R}^{n_1\times n_2\times\cdots\times n_d}$, the metric projection of $\tensA$ onto the tangent cone $\tangent_\tensX\!\tensM_{\leq\vecr}$ can be expressed by~\cite[\S 3.3]{gao2025low}
\begin{equation}
    \label{eq: exact projection}
    \proj_{\tangent_\tensX\!\tensM_{\leq\vecr}}\!(\tensA)=\tensA\times_{k=1}^d\proj_{\matS_{k}^*}+\sum_{k=1}^d\tensG\times_k\Big(\proj^\perp_{\matS_{k}^*}(\tensA\times_{j\neq k}^{}\matu_j^\top)_{(k)}\matG_{(k)}^\dagger\Big)\times_{j\neq k}\matu_j,
\end{equation}
where $\matS_{k}^*=[\matu_k\ \matu_{k,1}^*]$, $\proj_{\matS_{k}^*}=\matu_k^{}\matu_k^\top+\matu_{k,1}^*(\matu_{k,1}^*)^\top$, $\proj^\perp_{\matS_{k}^*}=\matI_{n_k}-\proj_{\matS_{k}^*}$, $(\matG_{(k)})^\dagger=\matG_{(k)}^\top(\matG_{(k)}^{}\matG_{(k)}^\top)^{-1}$, and $\{\matu_{k,1}^*\}_{k=1}^d$ is a global minimizer of a non-convex problem such that $\matu_k^\top\matu_{k,1}^*=0$, which is computationally intractable. Nevertheless, by choosing appropriate $\tilde{\matu}_{k,1}\in\St(r_k-\underline{r}_k,n_k)$ to substitute for $\matu_{k,1}^*$, we can construct two different search directions in tangent cones~\cite[\S 3.4 and \S 5.1]{gao2025low}. 

The first direction is an approximate projection
\begin{equation}
    \label{eq: an approximate projection}
    \approj_{\tangent_\tensX\!\tensM_{\leq\vecr}}\!(\tensA):=\tensA\times_{k=1}^d\proj_{\tilde{\matS}_{k}}+\sum_{k=1}^d\tensG\times_k\left(\proj^\perp_{\tilde{\matS}_{k}}\!\left(\tensA\times_{j\neq k}^{}\matu_j^\T\right)_{(k)}\matG_{(k)}^\dagger\right)\times_{j\neq k}\matu_j,
\end{equation}
where $\tilde{\matS}_{k}:=[\matu_k\ \tilde{\matu}_{k,1}]$. The other one is a partial projection
\begin{equation}
    \label{eq: retraction-free projection}
    \hat{\proj}_{\tangent_{\tensX}\!\tensM_{\leq\vecr}}(\tensA):=\argmax_{\tensV\in\{\tilde{\proj}_{0}(\tensA),\proj_1(\tensA)\dots,\proj_d(\tensA)}\}\|\tensV\|_\frob,
\end{equation}
where $\approj_{0}(\tensA):=\tensA\times_{k=1}^d\proj_{\tilde{\matS}_{k}}$, and $\proj_k(\tensA):=\tensG\times_k(\proj_{\matu_{k}}^\perp(\tensA\times_{j\neq k}\matu_j^\T)_{(k)}\matG_{(k)}^\dagger)\times_{j\neq k}\matu_j$ for $k\in[d]$. We observe that 
\[\tensX+\tilde{\proj}_0(\tensA)\in\tensM_{\leq\vecr}\quad\text{and}\quad\tensX+\proj_k(\tensA)\in\tensM_{\leq\underline{\vecr}},\]
i.e., $\tensX+\hat{\proj}_{\tangent_{\tensX}\!\tensM_{\leq\vecr}}(\tensA)\in\tensM_{\leq\vecr}$, which implies that the partial projection $\hat{\proj}_{\tangent_{\tensX}\!\tensM_{\leq\vecr}}$ is a \emph{retraction-free} search direction.

Note that when $\vecr=\ranktc(\tensX)$, the metric projection~\cref{eq: exact projection} boils down to the projection onto the tangent space which is of closed form, where $\tilde{\matu}_{k,1}\in\St(r_k-\underline{r}_k,n_k)$ is not involved. Alternatively, when $\vecr\neq\ranktc(\tensX)$, a specific $\tilde{\matu}_{k,1}$ is required to facilitate the projections~\cref{eq: an approximate projection,eq: retraction-free projection}. One can always construct the projections by choosing arbitrary $\tilde{\matu}_{k,1}$. Alternatively, we choose appropriate $\tilde{\matu}_{k,1}\in\St(r_k-\underline{r}_k,n_k)$ by singular value decomposition to achieve the \emph{angle condition}
\begin{align}
    \label{eq: angle condition GRAP}
    \langle\tensA,\approj_{\tangent_\tensX\!\tensM_{\leq\vecr}}(\tensA)\rangle&\geq\tilde{\omega}\,\|\proj_{\tangent_{\tensX}\!\tensM_{\leq\vecr}}(\tensA)\|_\frob\|\approj_{\tangent_\tensX\!\tensM_{\leq\vecr}}(\tensA)\|_\frob,\\
    \label{eq: angle condition rfGRAP}
    \langle\tensA,\hat{\proj}_{\tangent_{\tensX}\!\tensM_{\leq\vecr}}(\tensA)\rangle&\geq\hat{\omega}\,\|\proj_{\tangent_{\tensX}\!\tensM_{\leq\vecr}}(\tensA)\|_\frob\|\,\hat{\proj}_{\tangent_{\tensX}\!\tensM_{\leq\vecr}}(\tensA)\|_\frob
\end{align}
for all $\tensA\in\mathbb{R}^{n_1\times n_2\times\cdots\times n_d}$ and some $\tilde{\omega},\hat{\omega}\in(0,1]$ as follows. 

\begin{lemma}\label{lem: inequality}
    Given a tensor $\tensA\in\mathbb{R}^{n_1\times n_2\times \cdots\times n_d}$ and $\tensX\in\tensM_{\leq\vecr}$ with $\ranktc(\tensX)=\underline{\vecr}\neq \vecr$ and Tucker decomposition $\tensX=\tensG\times_{k=1}^d\matu_k$. Recall $I=\{k\in[d]:\underline{r}_k<r_k\}$. There exists $\tilde{\matu}_{k,1}\in\St(r_k-\underline{r}_k,n_k)$ for $k\in I$, such that $\matu_k^\top\tilde{\matu}_{k,1}^{}=0$ and
    \begin{equation}
        \label{eq: inequality}
        \|\tensA\times_{k\in I}\proj_{[\matu_k\ \tilde{\matu}_{k,1}]}\times_{k\in[d]\setminus I}\proj_{\matu_k}\|_\frob\geq\|\tensA\times_{k\in[d]\setminus I}\proj_{\matu_k}\|_\frob\prod_{k\in I}\sqrt{\frac{r_k-\underline{r}_k}{\min\{n_k,n_{-k}\}}}.
    \end{equation}
\end{lemma}
\begin{proof}
    We prove the result by recursively considering the left singular vectors corresponding to the leading $(r_k-\underline{r}_k)$ singular values of mode-$k$ unfolding matrices for $k\in I$. For the sake of brevity, we assume that $I=\{1,2,\dots,k_0\}$ for some $k_0\in[d]$. Specifically, we observe that there exists $\tilde{\matu}_{1,1}\in\St(r_1-\underline{r}_1,n_1)$, such that $\matu_1^\top\tilde{\matu}_{1,1}=0$ and $\Span([\matu_1\ \tilde{\matu}_{1,1}])$ contains the left singular vectors corresponding to the leading $(r_1-\underline{r}_1)$ singular values of $(\tensA\times_{k=k_0+1}^d\proj_{\matu_k})_{(1)}$. Therefore, we obtain that
    \[\|\tensA\times_{1}\proj_{[\matu_1\ \tilde{\matu}_{1,1}]}\times_{k=k_0+1}^d\proj_{\matu_k}\|_\frob\geq\|\tensA\times_{k=k_0+1}^d\proj_{\matu_k}\|_\frob\sqrt{\frac{r_1-\underline{r}_1}{\min\{n_1,n_{-1}\}}}.\]
    Subsequently, we can construct $\tilde{\matu}_{k,1}\in\St(r_k-\underline{r}_k,n_k)$ for $k=2,3,\dots,k_0$ recursively in a similar fashion and obtain that 
    \begin{align}
        &~~~~\|\tensA\times_{k=1}^{k_0}\proj_{[\matu_k\ \tilde{\matu}_{k,1}]}\times_{k=k_0+1}^d\proj_{\matu_k}\|_\frob^2\nonumber\\
        &\geq\|\tensA\times_{k=1}^{k_0-1}\proj_{[\matu_k\ \tilde{\matu}_{k,1}]}\times_{k=k_0+1}^d\proj_{\matu_k}\|_\frob^2\frac{r_{k_0}-\underline{r}_{k_0}}{\min\{n_{k_0},n_{-k_0}\}}\nonumber\\
        &\geq\|\tensA\times_{k=1}^{k_0-2}\proj_{[\matu_k\ \tilde{\matu}_{k,1}]}\times_{k=k_0+1}^d\proj_{\matu_k}\|_\frob^2\prod_{k=k_0-1}^{k_0}\frac{r_k-\underline{r}_k}{\min\{n_k,n_k\}}\nonumber\\
        &\geq\cdots\nonumber\\
        &\geq\|\tensA\times_{k=k_0+1}^d\proj_{\matu_k}\|_\frob^2\prod_{k=1}^{k_0}\frac{r_k-\underline{r}_k}{\min\{n_k,n_k\}}.\nonumber
    \end{align}
\end{proof}

It is worth noting that~\cref{eq: inequality} does not depend on a specific permutation of $I$. By using~\cite[Proposition 4]{gao2025low} and~\cref{lem: inequality}, we prove that the projections~\cref{eq: an approximate projection,eq: retraction-free projection} are orthogonal projections and satisfy the angle conditions~\cref{eq: angle condition GRAP,eq: angle condition rfGRAP}. Note that we choose the same $\tilde{\matu}_{k,1}$ for~\cref{eq: an approximate projection,eq: retraction-free projection}.

\begin{proposition}\label{prop: angle condition GRAP}
    Given $\tensX=\tensG\times_{k=1}^d\matu_k\in\tensM_{\underline{\vecr}}$ with $\underline{\vecr}\neq\vecr$. Recall $I=\{k\in[d]:\underline{r}_k<r_k\}$. The approximate projection~\cref{eq: an approximate projection} satisfies $\langle\tensA,\approj_{\tangent_\tensX\!\tensM_{\leq\vecr}}(\tensA)\rangle=\|\approj_{\tangent_\tensX\!\tensM_{\leq\vecr}}(\tensA)\|_\frob^2$ for all $\tensA\in\mathbb{R}^{n_1\times n_2\times \cdots\times n_d}$, and the angle condition~\cref{eq: angle condition GRAP} with
    \[
    \tilde{\omega}(\tensX)=\sqrt{\frac{1}{|I|+1}\prod_{k\in I}\frac{r_k-\underline{r}_k}{\min\{n_k,n_{-k}\}}}.\]
\end{proposition}
\begin{proof}
    Denote $c_I=\prod_{k\in I}(r_k-\underline{r}_k)/\min\{n_k,n_{-k}\}$. It follows from~\cref{eq: an approximate projection,eq: inequality} that
    \begin{align*}
        \|\approj_{\tangent_\tensX\!\tensM_{\leq\vecr}}(\tensA)\|_\frob^2     &=\|\tensA\times_{k=1}^d\proj_{\tilde{\matS}_{k}}\|_\frob^2+\sum_{k=1}^d\|\tensG\times_k(\proj^\perp_{\tilde{\matS}_{k}}\!(\tensA\times_{j\neq k}^{}\!\matu_j^\T)_{(k)}\matG_{(k)}^\dagger)\times_{j\neq k}\!\matu_j\|_\frob^2\\
        &\geq \|\tensA\times_{k=1}^d\proj_{\tilde{\matS}_{k}}\|_\frob^2+\sum_{k\in[d]\setminus I}\|\tensG\times_k(\proj^\perp_{\matu_{k}}\!(\tensA\times_{j\neq k}^{}\matu_j^\T)_{(k)}\matG_{(k)}^\dagger)\|_\frob^2\\
        &\geq c_I\|\tensA\times_{k\in[d]\setminus I}\proj_{\matu_k}\|_\frob^2+\sum_{k\in[d]\setminus I}\|\tensG\times_k(\proj^\perp_{\matu_{k}}\!(\tensA\times_{j\neq k}^{}\matu_j^\T)_{(k)}\matG_{(k)}^\dagger)\|_\frob^2\\
        &\geq \frac{c_I}{|I|+1}\|\proj_{\tangent_\tensX\!\tensM_{\leq\vecr}}\!\tensA\|_\frob^2,
    \end{align*}
    where the first inequality follows from omitting the terms with $k\in I$ and the orthogonality of $\matu_k$, and the last inequality follows from 
        \begin{align*}
            \|\tensA\times_{k\in[d]\setminus I}\proj_{\matu_k}\|_\frob&\geq\|\tensA\times_{k\in I}\proj_{\matS_k^*}\times_{k\in[d]\setminus I}\proj_{\matu_k}\|_\frob,\\
            \|\tensA\times_{j\in[d]\setminus I}\proj_{\matu_j}\|_\frob&\geq\|\tensA\times_{j\neq k}^{}\proj_{\matu_j}\|_\frob=\|(\tensA\times_{j\neq k}^{}\matu_j^\T)_{(k)}\|_\frob\\
            &\geq\|\tensG\times_k(\proj^\perp_{\matS_{k}^*}\!(\tensA\times_{j\neq k}^{}\matu_j^\T)_{(k)}\matG_{(k)}^\dagger)\times_{j\neq k}\matu_j\|_\frob
        \end{align*} 
    for all $k\in I$, and $\|\matG_{(k)}^\dagger\matG_{(k)}^{}\|\leq 1$.
\end{proof}

Additionally, we prove that the partial projection~\cref{eq: retraction-free projection} satisfies the angle condition~\cref{eq: angle condition rfGRAP} in the following proposition.
\begin{proposition}\label{prop: approximate rfGRAP}
    Given $\tensX=\tensG\times_{k=1}^d\matu_k\in\tensM_{\underline{\vecr}}$ with $\underline{\vecr}\neq\vecr$. Recall $I=\{k\in[d]:\underline{r}_k<r_k\}$. The partial projection~\cref{eq: retraction-free projection} satisfies $\langle\tensA,\hat{\proj}_{\tangent_\tensX\!\tensM_{\leq\vecr}}(\tensA)\rangle=\|\hat{\proj}_{\tangent_\tensX\!\tensM_{\leq\vecr}}(\tensA)\|_\frob^2$ for all $\tensA\in\mathbb{R}^{n_1\times n_2\times \cdots\times n_d}$, and the angle condition~\cref{eq: angle condition rfGRAP} with
    \[ 
    \hat{\omega}(\tensX)=\sqrt{\frac{1}{d+1}\prod_{k\in I}\frac{r_k-\underline{r}_k}{\min\{n_k,n_{-k}\}}}.
    \]
\end{proposition}
\begin{proof}
    Recall $c_I=\prod_{k\in I}(r_k-\underline{r}_k)/\min\{n_k,n_{-k}\}$. By using~\cref{eq: inequality}, we obtain that
    \begin{align*}
        &~~~~\|\hat{\proj}_{\tangent_{\tensX}\!\tensM_{\leq\vecr}}(\tensA)\|_\frob^2=\max\{\|\approj_0(\tensA)\|_\frob^2,\max_{k=1,2,\dots,d}\|\proj_k(\tensA)\|_\frob^2\}\\
        &\geq\frac{1}{d+1}\Big(\|\approj_0(\tensA)\|_\frob^2+\sum_{k=1}^d\|\proj_k(\tensA)\|_\frob^2\Big)\\
        &=\frac{1}{d+1}\big(\| \tensA\times_{k=1}^d\proj_{\tilde{\matS}_{k}}\|_\frob^2+\sum_{k=1}^d\|\tensG\times_k(\proj_{\matu_{k}}^\perp\!( \tensA\times_{j\neq k}\matu_j^\T)_{(k)}\matG_{(k)}^\dagger)\times_{j\neq k}\matu_j\|_\frob^2\big)\\
        &\geq\frac{1}{d+1}\big(c_I\| \tensA\times_{k\in[d]\setminus I}\proj_{\matu_k}\|_\frob^2+\sum_{k=1}^d\|\tensG\times_k(\proj_{{\matS}^*_{k}}^\perp(\tensA\times_{j\neq k}\matu_j^\T)_{(k)}\matG_{(k)}^\dagger)\|_\frob^2\big)\\
            &\geq\frac{1}{d+1}\big(c_I\| \tensA\times_{k=1}^d\proj_{{\matS}^*_{k}}\|_\frob^2+\sum_{k=1}^d\|\tensG\times_k(\proj_{{\matS}^*_{k}}^\perp(\tensA\times_{j\neq k}\matu_j^\T)_{(k)}\matG_{(k)}^\dagger)\times_{j\neq k}\matu_j\|_\frob^2\big)\\
        &=\frac{1}{d+1}\prod_{k\in I}\frac{r_k-\underline{r}_k}{\min\{n_k,n_{-k}\}}\|\proj_{\tangent_\tensX\!\tensM_{\leq\vecr}}(\tensA)\|_\frob^2,
    \end{align*}
    where we use the facts $\matS^*_k=[\matu_k\ \matu_{k,1}^*]$ if $\underline{r}_k<r_k$, $\matS^*_k=\matu_k$ if $\underline{r}_k=r_k$, and $\Span(\matS^*_k)^\perp\subseteq\Span(\matu_k)^\perp$.
\end{proof}


\section{Gradient-related\,approximate projection method with rank decrease (GRAP-R)}\label{sec: GRAP-R}
We propose a first-order method for~\cref{eq: problem (P)} by combining the approximate projection~\cref{eq: an approximate projection} with a rank-decreasing mechanism. We prove that the GRAP-R method converges to stationary points via a sufficient decrease condition.

\subsection{Algorithm}
The GRAP-R method is depicted in~\cref{alg: GRAP-R}, which consists of a rank-decreasing mechanism and line search. Specifically, given $\tensX^{(t)}\in\tensM_{\leq\vecr}$, we monitor the singular values of unfolding matrices $\matx_{(1)}^{(t)},\matx_{(2)}^{(t)},\dots,\matx_{(d)}^{(t)}$, and construct the index sets 
\begin{equation}
    \label{eq: GRAP-R index}
    I_k=\{\rank_\Delta(\matx_{(k)}^{(t)}),\rank_\Delta(\matx_{(k)}^{(t)})+1,\dots,\rank(\matx_{(k)}^{(t)})-1,\rank(\matx_{(k)}^{(t)})\}
\end{equation}
for $k\in[d]$. These index sets are designed to identify the singular values smaller than a threshold $\Delta>0$ and prevent the removal of small singular values that still contribute to the underlying structure of a tensor. We project $\tensX^{(t)}$ to $\tensM_{\underline{\vecr}}$ by~HOSVD in~\cref{eq: HOSVD} and yield lower-rank candidates $\tensX_{\underline{\vecr}}^{(t)}:=\proj_{\underline{\vecr}}^{\mathrm{HO}}(\tensX^{(t)})$ for all $\underline{\vecr}\in I_1\times I_2\times\cdots\times I_d$. Note that in the worst case $\rank_\Delta(\matx_{(k)}^{(t)})=0$, $(\prod_{k=1}^d r_k^{(t)}+1)$ lower-rank candidates are generated, where $r_k^{(t)}=\rank(\matx_{(k)}^{(t)})$. Therefore, when the rank parameters $r_k^{(t)}$ become large, the number of candidates can grow exponentially, rendering the GRAP-R method become less efficient in the high-rank settings, as observed in our numerical experiments in~\cref{sec: experiments}.

Subsequently, we perform line searches for all lower-rank candidates $\tensX_{\underline{\vecr}}^{(t)}$ via 
\begin{equation}
    \label{eq: GRAP mapping}
    \tensY_{\underline{\vecr}}^{(t)}=\proj_{\leq\vecr}^{\mathrm{HO}}(\tensX_{\underline{\vecr}}^{(t)}+s_{\underline{\vecr}}^{(t)}\approj_{\tangent_{\tensX_{\underline{\vecr}}^{(t)}}\!\tensM_{\leq\vecr}}(-\nabla f(\tensX_{\underline{\vecr}}^{(t)})))
\end{equation}
based on the approximate projection~\cref{eq: an approximate projection}.
The stepsize $s^{(t)}_{\underline{\vecr}}$ is chosen by the Armijo backtracking line search: given an initial stepsize $\bar{s}>0$, $\tensX\in\tensM_{\leq\vecr}$ and a search direction $\tensV\in\tangent_\tensX\!\tensM_{\leq\vecr}$, find the smallest integer~$l$, such that for $s=\rho^l 
\bar{s}$, the inequality 
\begin{equation}
    \label{eq: Armijo stepsize}
    f(\tensX)-f(\proj_{\leq\vecr}^{\mathrm{HO}}(\tensX+s\tensV))\geq s a \langle-\nabla f(\tensX),\tensV\rangle
\end{equation}
holds, where $\rho,\ a\in(0,1)$ are backtracking parameters. Finally, the point $\tensX^{(t+1)}$ is updated by choosing the point in $\{\tensY^{(t)}_{\underline{\vecr}}\}_{\underline{\vecr}}$ that has the lowest function value. The algorithm terminates if $\|\approj_{\tangent_{\tensX^{(t)}}\!\tensM_{\leq\vecr}}(-\nabla f(\tensX^{(t)}))\|_\frob$ vanishes, which implies that $\tensX^{(t)}$ is a stationary point of~\cref{eq: problem (P)}, since the angle condition~\cref{eq: angle condition GRAP} provides an upper bound to the metric projection $\proj_{\tangent_{\tensX^{(t)}}\!\tensM_{\leq\vecr}}(-\nabla f(\tensX^{(t)}))$.

\begin{algorithm}[htbp]
    \caption{Gradient-related approximate projection method with rank decrease (GRAP-R)}
    \label{alg: GRAP-R}
    \begin{algorithmic}[1]
        \REQUIRE Initial guess $\tensX^{(0)}\in\tensM_{\leq\vecr}$, $\Delta>0$, backtracking parameters $\rho, a\in(0,1)$, $t=0$.\\      
        \WHILE{The stopping criteria are not satisfied}
            \STATE Compute the index sets $I_k$ for $k\in[d]$ via~\cref{eq: GRAP-R index}.
            \FOR{$\underline{\vecr}\in I_1\times I_2\times\cdots\times I_d$}
                \STATE Compute a low-rank candidate $\proj_{\underline{\vecr}}^{\mathrm{HO}}(\tensX^{(t)})$ by~\cref{eq: HOSVD}. 
                \STATE Compute $\tensY_{\underline{\vecr}}^{(t)}$ by~\cref{eq: GRAP mapping} with Armijo backtracking line search~\cref{eq: Armijo stepsize}.
            \ENDFOR
            \STATE Update $\tensX^{(t+1)}=\argmin_{\underline{\vecr}\in I_1\times I_2\times\cdots\times I_d}f(\tensY_{\underline{\vecr}}^{(t)})$ and $t=t+1$. \label{line: GRAP-R update}
        \ENDWHILE
        \ENSURE $\tensX^{(t)}$.
    \end{algorithmic}
\end{algorithm}

\subsection{Backtracking line search}
In order to facilitate the global convergence, we aim to provide a lower bound $\underline{s}$ of the stepsize in~\cref{eq: Armijo stepsize}. Given a closed ball $B\subsetneq\mathbb{R}^{n_1\times n_2\times\cdots\times n_d}$, we assume that $\nabla f$ is locally Lipschitz continuous with constant
\[L_{B}:=\sup_{\tensX,\tensY\in B,\ \tensX\neq\tensY}\frac{\|\nabla f(\tensX)-\nabla f(\tensY)\|_\frob}{\|\tensX-\tensY\|_\frob}<\infty.\]
Therefore, it holds that
\begin{equation}
    \label{eq: Lipschitz}
    |f(\tensY)-f(\tensX)-\langle\nabla f(\tensX),\tensY-\tensX\rangle|\leq\frac{L_{B}}{2}\|\tensY-\tensX\|_\frob^2.
\end{equation}
First, we prove the following lemma for the metric projection.
\begin{lemma}
    Given $\tensX\in\tensM_{\leq\vecr}\setminus\{0\}$ and $\tensV\in\tangent_{\tensX}\!\tensM_{\leq\vecr}\setminus\{0\}$, it holds that
    \begin{equation}
        \label{eq: curvature}
        \|\proj_{\leq\vecr}(\tensX+s\tensV)-(\tensX+s\tensV)\|_\frob\leq c(\tensX) s^2\|\tensV\|_\frob^2\quad\text{with}\quad c(\tensX)=(\frac{1}{\sigma_{\min}(\tensX)^2}+1)^{\frac{d}{2}}
    \end{equation}
    for $s\in[0,1/\|\tensV\|_\frob]$, where $\sigma_{\min}(\tensX):=\min_{k\in [d]}\sigma_{\min}(\matx_{(k)})$. 
\end{lemma}
\begin{proof}
    Given $\tensV\in\tangent_{\tensX}\!\tensM_{\leq\vecr}$ and its representation~\cref{eq: Tucker tangent cone}, we consider the smooth curve 
    \[\gamma(s)=s\tensC\times_{k=1}^d\begin{bmatrix}
        \matu_k+s\dot{\matu}_k &\ \matu_{k,1}
    \end{bmatrix}+\tensG\times_{k=1}^d(\matu_k+s\dot{\matu}_k).\]
    It follows from~\cref{eq: Tucker tangent cone} that $\gamma(0)=\tensX$, $\gamma^\prime(0)=\tensV$, and the norm of $\tensV$ is $\|\tensV\|_\frob^2=\|\tensC\|_\frob^2+\sum_{k=1}^{d}\|\tensG\times_k\dot{\matu}_k\|_\frob^2$. Since $\matG_{(k)}$ is of full rank, we obtain that 
    \[\|\dot{\matu}_k\|_\frob\leq\|\dot{\matu}_k\matG_{(k)}\|_\frob\|\matG_{(k)}^\dagger\|_2\leq\frac{1}{\sigma_{\min}(\matG_{(k)})}\|\tensV\|_\frob\leq\frac{1}{\sigma_{\min}(\tensX)}\|\tensV\|_\frob,\]
    where $\|\matG_{(k)}^\dagger\|_2:=\sigma_1(\matG_{(k)}^\dagger)$ and we use the fact $\sigma_{\min}(\tensX)=\sigma_{\min}(\tensG)$ by orthogonality of~$\matu_k$. Therefore, by using the Taylor expansion of $\gamma(s)$, the orthogonality of $\matu_k,\matu_{k,1}$, and $[\matu_k\ \matu_{k,1}]^\top\dot{\matu}_k=0$, we obtain that
    \begin{align*}
        &~~~~\|\proj_{\leq\vecr}(\tensX+s\tensV)-(\tensX+s\tensV)\|_\frob^2\leq\|\gamma(s)-(\tensX+s\tensV)\|_\frob^2\\
        &=s^4\sum_{k=1}^{d}\Big(\|\tensC\times_k[\dot{\matu}_k\ 0]\|_\frob^2+\sum_{j\neq k}\|\tensG\times_j\dot{\matu}_j\times_k\dot{\matu}_k\|_\frob^2\Big)\\
        &~~~+s^6\sum_{j,k=1,j\neq k}^{d}\Big(\|\tensC\times_j[\dot{\matu}_j\ 0]\times_k[\dot{\matu}_k\ 0]\|_\frob^2+\sum_{l=1,l\neq j,l\neq k}^{d}\|\tensG\times_j\dot{\matu}_j\times_k\dot{\matu}_k\times_l\dot{\matu}_l\|_\frob^2\Big)\\
        &~~~+\cdots\\
        &~~~+s^{2(d+1)}\|\tensC\times_{k=1}^d[\dot{\matu}_k\ 0]\|_\frob^2\\
        &\leq s^4\sum_{k=1}^{d}\binom{d}{k}\frac{s^{2k-2}\|\tensV\|_\frob^{2k+2}}{\sigma_{\min}(\tensX)^{2k}}\\
        &\leq s^4\|\tensV\|_\frob^4(\frac{1}{\sigma_{\min}(\tensX)^2}+1)^d,
    \end{align*}
    where $\binom{d}{k}:=d!/(k!(d-k)!)$ is the combinatorial number and we use the fact $s\|\tensV\|_\frob\leq 1$ in the last inequality.
\end{proof}

We prove the following lemma for the HOSVD~\cref{eq: HOSVD}.
\begin{lemma}\label{lem: HOSVD inequality}
    For all $\tensX\in\tensM_{\leq\vecr}$ and $\tensA\in\mathbb{R}^{n_1\times n_2\times\cdots\times n_d}$, the HOSVD satisfies that $\|\proj_{\leq\vecr}^\mathrm{HO}(\tensA)-\tensX\|_\frob\leq (\sqrt{d}+1)\|\tensX-\tensA\|_\frob$. 
\end{lemma}
\begin{proof}
    It follows from~\eqref{eq: quasi-optimal HOSVD} that
        \begin{align*}
            \|\proj_{\leq\vecr}^\mathrm{HO}(\tensA)-\tensX\|_\frob&\leq\|\proj_{\leq\vecr}^\mathrm{HO}(\tensA)-\tensA\|_\frob+\|\tensA-\tensX\|_\frob\\
            &\leq\sqrt{d}\|\proj_{\leq\vecr}(\tensA)-\tensA\|_\frob+\|\tensA-\tensX\|_\frob\leq(\sqrt{d}+1)\|\tensA-\tensX\|_\frob,
        \end{align*}
    where the last inequality comes from $\tensX\in\tensM_{\leq\vecr}$.
\end{proof}

Subsequently, we develop a lower bound for the stepsize in Armijo condition~\cref{eq: Armijo stepsize}. For the sake of brevity, we denote the metric projection $\tensV_\tensX=\proj_{\tangent_{\tensX}\!\tensM_{\leq\vecr}}(-\nabla f(\tensX))$ and the approximate projection $\tilde{\tensV}_\tensX=\approj_{\tangent_{\tensX}\!\tensM_{\leq\vecr}}(-\nabla f(\tensX))$ in~\cref{eq: an approximate projection}.
\begin{proposition}\label{prop: backtracking GRAP}
    Given a point $\tensX\in\tensM_{\leq\vecr}$ satisfying $\|\tensV_\tensX\|_\frob\neq 0$ and $\bar{s}\in(0,\min\{1/\|\tensV_\tensX\|_\frob,1\}]$. Consider the closed ball $B_\tensX=B[\tensX,(\sqrt{d}+1)\bar{s}\|\tensV_\tensX\|_\frob]$. Then, it holds that $\proj^{\ho}_{\leq\vecr}(\tensX+s\tilde{\tensV}_\tensX)\in B_\tensX$ and the Armijo condition
    \begin{equation}
        \label{eq: Armijo GRAP}
        f(\tensX)-f(\proj_{\leq\vecr}^{\mathrm{HO}}(\tensX+s\tilde{\tensV}_\tensX))\geq s a \tilde{\omega}(\tensX)^2\|\tensV_\tensX\|_\frob^2
    \end{equation}
    for some $s\in[\underline{s}(\tensX,B_\tensX),\bar{s}]$, where $\underline{s}(\tensX,B_\tensX)=\min\{\rho(1-a)/\tilde{a}(\tensX,B_\tensX),\rho\bar{s}\}$ and
    \[\tilde{a}(\tensX,B_\tensX)=\left\{\begin{array}{ll}
        \frac12 L_{B_\tensX} & \text{if}\ \tensX=0,\\
        \|\nabla f(\tensX)\|_\frob\sqrt{d} c(\tensX)+\frac{L_{B_\tensX}}{2}(\sqrt{d}+1)^2 & \text{otherwise.}
    \end{array}\right.\]
\end{proposition}
\begin{proof}
    First, we prove that 
    \begin{equation}
        \label{eq: GRAP sufficient decrease}
        f(\proj^{\ho}_{\leq\vecr}(\tensX+s\tilde{\tensV}_\tensX))\leq f(\tensX)+\|\tilde{\tensV}_\tensX\|_\frob^2 s(-1+s\tilde{a}(\tensX,B_\tensX))
    \end{equation}
    for all $s\in[0,\bar{s}]$. It follows from~\cref{lem: HOSVD inequality} that 
    \begin{equation*}
        \|\proj^{\ho}_{\leq\vecr}(\tensX+s\tilde{\tensV}_\tensX)-\tensX\|_\frob\leq(\sqrt{d}+1)s\|\tilde{\tensV}_\tensX\|_\frob\leq(\sqrt{d}+1)\bar{s}\|\tensV_\tensX\|_\frob.
    \end{equation*}
    If $\tensX=0$, it holds that $\tilde{\tensV}_{\tensX}\in\tangent_{0}\!\tensM_{\leq\vecr}=\tensM_{\leq\vecr}$ and thus~\cref{eq: GRAP sufficient decrease} immediately follows from~\cref{eq: Lipschitz}. Otherwise, we consider the case $\tensX\neq 0$. Since $s\|\tilde{\tensV}_\tensX\|_\frob\leq s\|\tensV_\tensX\|_\frob\leq 1$, it follows from~\cref{eq: Lipschitz,eq: curvature} that
        \begin{align*}
            &~~~~f(\proj^{\ho}_{\leq\vecr}(\tensX+s\tilde{\tensV}_\tensX))-f(\tensX)\\
            &\leq\langle\nabla f(\tensX),\proj^{\ho}_{\leq\vecr}(\tensX+s\tilde{\tensV}_\tensX)-\tensX\rangle+\frac{L_{B_\tensX}}{2}\|\proj^{\ho}_{\leq\vecr}(\tensX+s\tilde{\tensV}_\tensX)-\tensX\|_\frob^2\\
            &=\langle\nabla f(\tensX),\proj^{\ho}_{\leq\vecr}(\tensX+s\tilde{\tensV}_\tensX)-(\tensX+s\tilde{\tensV}_\tensX)+s\tilde{\tensV}_\tensX\rangle+\frac{L_{B_\tensX}}{2}\|\proj^{\ho}_{\leq\vecr}(\tensX+s\tilde{\tensV}_\tensX)-\tensX\|_\frob^2\\
            &\leq\|\nabla f(\tensX)\|_\frob\sqrt{d}\|\proj_{\leq\vecr}(\tensX+s\tilde{\tensV}_\tensX)\!-\!(\tensX+s\tilde{\tensV}_\tensX)\|_\frob \!-\! s\|\tilde{\tensV}_\tensX\|_\frob^2+\frac{L_{B_\tensX}}{2}(\sqrt{d}+1)^2s^2\|\tilde{\tensV}_\tensX\|_\frob^2\\
            &\leq\|\nabla f(\tensX)\|_\frob\sqrt{d} c(\tensX) s^2\|\tilde{\tensV}_\tensX\|_\frob^2-s\|\tilde{\tensV}_\tensX\|_\frob^2+\frac{L_{B_\tensX}}{2}(\sqrt{d}+1)^2s^2\|\tilde{\tensV}_\tensX\|_\frob^2\\
            &\leq s\|\tilde{\tensV}_\tensX\|_\frob^2(-1+\|\nabla f(\tensX)\|_\frob\sqrt{d} c(\tensX) s+\frac{L_{B_\tensX}}{2}(\sqrt{d}+1)^2s)\\
            &=s\|\tilde{\tensV}_\tensX\|_\frob^2(-1+s\tilde{a}(\tensX,B_\tensX)),
        \end{align*}
    where $\tilde{a}(\tensX,B_\tensX)=\|\nabla f(\tensX)\|_\frob\sqrt{d} c(\tensX)+\frac{L_{B_\tensX}}{2}(\sqrt{d}+1)^2$ and the second inequality follows from $\langle\nabla f(\tensX),\tilde{\tensV}_\tensX\rangle=-\|\tilde{\tensV}_\tensX\|_\frob^2$ in~\cref{prop: angle condition GRAP}, Cauchy--Schwarz inequality and the quasi-optimality~\cref{eq: quasi-optimal HOSVD}.

    Subsequently, we observe that 
    \[s\|\tilde{\tensV}_\tensX\|_\frob^2(-1+s\tilde{a}(\tensX,B_\tensX))\leq -s a \|\tilde{\tensV} _\tensX\|_\frob^2\leq -s a \tilde{\omega}(\tensX)^2 \|\tensV_\tensX\|_\frob^2\] 
    if and only if $s\leq(1-a)/\tilde{a}(\tensX,B_\tensX)$. Therefore, \cref{eq: Armijo stepsize} can be satisfied for all $s\in[\rho(1-a)/\tilde{a}(\tensX,B_\tensX),(1-a)/\tilde{a}(\tensX,B_\tensX)]$. Note that if $\bar{s}\leq(1-a)/\tilde{a}(\tensX,B_\tensX)$, the Armijo condition~\cref{eq: Armijo stepsize} holds for $s\in[\rho\bar{s},\bar{s}]$. Consequently, the Armijo condition~\cref{eq: Armijo stepsize} holds for some $s\in[\underline{s}(\tensX,B_\tensX),\bar{s}]$, where \[\underline{s}(\tensX,B_\tensX)=\min\{\rho(1-a)/\tilde{a}(\tensX,B_\tensX),\rho\bar{s}\}.\]
\end{proof}

\subsection{Convergence to stationary points}\label{subsec: GRAP-R convergence}
Let $\{\tensX^{(t)}\}$ be a sequence generated by GRAP-R method in~\cref{alg: GRAP-R}. We show that all accumulation points of $\{\tensX^{(t)}\}$ are stationary. To this end, we validate the following sufficient decrease condition, which implies that one step of the GRAP-R method---let $\tensX=\tensX^{(t)}$ and $\tensY=\tensX^{(t+1)}$---is guaranteed to decrease the function value when stationarity is not attained.

\begin{proposition}\label{prop: GRAP-R decrease}
    For all $\underline{\tensX}\in\tensM_{\leq\vecr}$ and $\|\tensV_{\underline{\tensX}}\|_\frob\neq 0$, there exists $\varepsilon(\underline{\tensX})>0$ and $\delta(\underline{\tensX})>0$, such that $f(\tensY)-f(\tensX)<\delta(\underline{\tensX})$ holds for all $\tensX\in B[\underline{\tensX},\varepsilon(\underline{\tensX})]\cap\tensM_{\leq\vecr}$ and $\tensY$ generated by GRAP-R in~\cref{alg: GRAP-R} for one step with the initial guess $\tensX$.
\end{proposition}
\begin{proof}
    Since the proof is technical, we give a proof sketch in~\cref{fig: proof sketch GRAP-R}. Generally speaking, given $\underline{\tensX}\in\tensM_{\leq\vecr}$ with $\underline{\vecr}=\ranktc(\underline{\tensX})$, if $\underline{\vecr}=\vecr$, we prove the result by the continuity of the projected anti-gradient. Otherwise, we project $\tensX$ onto $\tensM_{\underline{\vecr}}$, and adopt line search in~\cref{eq: GRAP mapping} to yield $\tensY_{\underline{\vecr}}$. Consequently, we obtain that 
    \[f(\tensY)\leq f(\tensY_{\underline{\vecr}})\leq f(\proj^{\mathrm{HO}}_{\underline{\vecr}}(\tensX))-3\delta(\underline{\tensX})\leq f(\tensX)-\delta(\underline{\tensX})\]
    by continuity of $f$ for $\tensX\in B[\underline{\tensX},\varepsilon(\underline{\tensX})]\cap\tensM_{\leq\vecr}$ with an appropriate $\varepsilon(\underline{\tensX})$.
    
    \begin{figure}[htbp]
        \centering
        \begin{tikzpicture}
            \def\h{2};
            \node[draw,rounded corners] (O) at (0,0) {Goal: $f(\tensY)-f(\tensX)\leq-\delta(\underline{\tensX})<0$, for all $\|\tensX-\underline{\tensX}\|_\frob\leq\varepsilon(\underline{\tensX})$};

            \node[draw,rounded corners,inner sep=10pt] (T) at ($(O)-(0,2)$) {\begin{minipage}{10cm}
                \centering\vspace{-2mm}
                Three inequalities\vspace{5mm}\\
                $f(\tensY)\quad\leq\quad f(\tensY_{\underline{\vecr}})\quad\leq\quad f(\proj^{\mathrm{HO}}_{\underline{\vecr}}(\tensX))-3\delta(\underline{\tensX})\quad\leq\quad f(\tensX)-\delta(\underline{\tensX})$
            \end{minipage}};

            \draw[->,arrows = {-Latex[width=5pt, length=5pt]}] (T) -- (O);

            \coordinate (C) at ($(T)-(3.6,0.3)$);
            \def\x{1.45};
            \def\y{0.25};
            \draw[densely dotted] ($(C)+(\x,\y)$) -- ($(C)+(\x,-\y)$) -- ($(C)+(-\x,-\y)$) -- ($(C)+(-\x,\y)$) -- cycle;
            
            \node[draw,densely dotted] (F) at ($(T)+(-4.5,-\h)$) {Line~\ref{line: GRAP-R update} in~\cref{alg: GRAP-R}};
            \draw[->,arrows = {-Latex[width=5pt, length=5pt]}] (F) -- ($(C)-(0,\y)$);

            \coordinate (C) at ($(T)-(0.5,0.3)$);
            \def\x{2.6};
            \def\y{0.4};
            \draw[dashed,rounded corners] ($(C)+(\x,\y)$) -- ($(C)+(\x,-\y)$) -- ($(C)+(-\x,-\y)$) -- ($(C)+(-\x,\y)$) -- cycle;

            \node[draw,dashed,rounded corners] (S) at ($(T)+(0,-\h)$) {Armijo condition~\cref{eq: Armijo GRAP}};
            \draw[->,arrows = {-Latex[width=5pt, length=5pt]}] (S) -- ($(T)-(0,\y)-(0,0.3)$);

            \coordinate (C) at ($(T)+(1.95,-0.3)$);
            \def\x{3.1};
            \def\y{0.32};
            \draw[densely dotted] ($(C)+(\x,\y)$) -- ($(C)+(\x,-\y)$) -- ($(C)+(-\x,-\y)$) -- ($(C)+(-\x,\y)$) -- cycle;

            \node[draw,densely dotted] (Third) at ($(T)+(4.5,-\h)$) {\cref{lem: HOSVD inequality}};
            \draw[->,arrows = {-Latex[width=5pt, length=5pt]}] (Third) -- ($(C)+(0.5*\x,-\y)$);
            \draw[->,dashed,arrows = {-Latex[width=5pt, length=5pt]}] (Third) -- (S);

            \node[draw,densely dotted] (Lem1) at ($(S)+(-2,-0.7*\h)$) {Angle condition~\cref{eq: angle condition GRAP}};
            \draw[->,dashed,arrows = {-Latex[width=5pt, length=5pt]}] (Lem1) -- (S);
            
            \node[draw,densely dotted] (Lem2) at ($(S)+(1.3,-0.7*\h)$) {Bound~\cref{eq: curvature}};
            \draw[->,dashed,arrows = {-Latex[width=5pt, length=5pt]}] (Lem2) -- (S);

            \node[draw,densely dotted] (Lem3) at ($(Third)+(0,-0.7*\h)$) {Quasi-optimality~\cref{eq: quasi-optimal HOSVD}};
            \draw[->,dashed,arrows = {-Latex[width=5pt, length=5pt]}] (Lem3) -- (Third);

        \end{tikzpicture}
        \caption{Proof sketch of~\cref{prop: GRAP-R decrease}.}
        \label{fig: proof sketch GRAP-R}
    \end{figure}
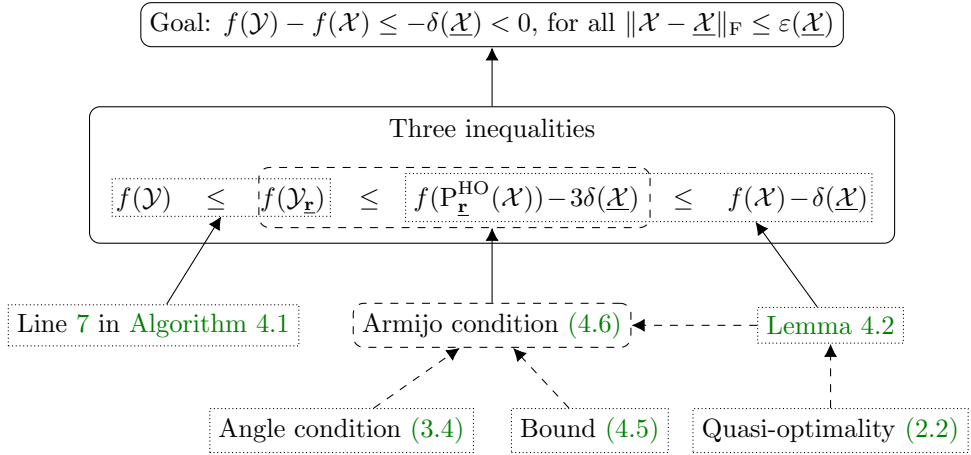

    First, we introduce auxiliary variables. Since $\nabla f$ is continuous and $\|\nabla f(\underline{\tensX})\|_\frob\neq 0$, there exists $\varepsilon_1(\underline{\tensX})>0$, such that 
    \[\|\nabla f(\tensX)\|_\frob\in[\frac12\|\nabla f(\underline{\tensX})\|_\frob,\frac{3}{2}\|\nabla f(\underline{\tensX})\|_\frob]\]
    for all $\tensX\in B[\underline{\tensX},\varepsilon_1(\underline{\tensX})]$. For all $\tensZ\in B[\tensX,\sqrt{d+1}\|\tensV_\tensX\|_\frob]$, it holds that
    \begin{align*}
        \|\tensZ-\underline{\tensX}\|_\frob&\leq\|\tensZ-\tensX\|_\frob+\|\tensX-\underline{\tensX}\|_\frob\leq \sqrt{d+1}\|\tensV_\tensX\|_\frob + \varepsilon_1(\underline{\tensX})\\
        &\leq \sqrt{d+1}\|\nabla f(\tensX)\| + \varepsilon_1(\underline{\tensX})\leq \bar{\varepsilon}(\underline{\tensX}):=\frac{3\sqrt{d+1}}{2}\|\nabla f(\underline{\tensX})\|_\frob+\varepsilon_1(\underline{\tensX}).
    \end{align*}
    Therefore, we obtain that 
    \[B[\tensX,\sqrt{d+1}\bar{s}\|\tensV_\tensX\|_\frob]\subseteq B[\tensX,\sqrt{d+1}\|\tensV_\tensX\|_\frob]\subseteq B[\underline{\tensX},\bar{\varepsilon}(\underline{\tensX})]\] 
    holds for all $\tensX\in B[\underline{\tensX},\varepsilon_1(\underline{\tensX})]$.

    Subsequently, we take the auxiliary variables into the Armijo condition in~\cref{eq: Armijo GRAP}. We adopt~\cref{eq: Armijo GRAP} to $\tensX\in B[\underline{\tensX},\varepsilon_1(\underline{\tensX})]\cap\tensM_{\leq\vecr}$ and closed ball $B[\tensX,\sqrt{d+1}\bar{s}\|\tensV_\tensX\|_\frob]$ and obtain that
        \begin{align*}
            f(\tensX)-f(\tilde{\tensX})&\geq a \underline{s}(\tensX,B[\tensX,\sqrt{d+1}\bar{s}\|\tensV_\tensX\|_\frob])\tilde{\omega}^2 \|\tensV_\tensX\|_\frob^2\\
            &\geq a \tilde{\omega}^2\|\tensV_\tensX\|_\frob^2\min\{\frac{\rho(1-a)}{\tilde{a}(\tensX,B[\underline{\tensX},\bar{\varepsilon}(\underline{\tensX})])},\rho,\frac{\rho}{\|\tensV_\tensX\|_\frob}\}\\
            &\geq a\tilde{\omega}^2\|\tensV_\tensX\|_\frob^2\min\{\frac{\rho(1-a)}{\tilde{a}(\tensX,B[\underline{\tensX},\bar{\varepsilon}(\underline{\tensX})])},\rho,\frac{\rho}{\|\nabla f(\tensX)\|_\frob}\}\\
            &\geq a\tilde{\omega}^2\|\tensV_\tensX\|_\frob^2\min\{\frac{\rho(1-a)}{\tilde{a}(\tensX,B[\underline{\tensX},\bar{\varepsilon}(\underline{\tensX})])},\rho,\frac{2\rho}{3\|\nabla f(\underline{\tensX})\|_\frob}\}.
        \end{align*}
    where $\tilde{\tensX}=\proj_{\leq\vecr}^{\mathrm{HO}}(\tensX+\underline{s}(\tensX,B[\tensX,\sqrt{d+1}\bar{s}\|\tensV_\tensX\|_\frob])\tilde{\tensV}_\tensX)$. 
    
    We aim to provide a lower bound of $\rho(1-a)/\tilde{a}(\tensX,B[\underline{\tensX},\bar{\varepsilon}(\underline{\tensX})])$, or equivalently an upper bound of $\tilde{a}(\tensX,B[\underline{\tensX},\bar{\varepsilon}(\underline{\tensX})])$, from the terms only depending on $\underline{\tensX}$. Recall that
    \[\tilde{a}(\tensX,B[\underline{\tensX},\bar{\varepsilon}(\underline{\tensX})])=\left\{\begin{array}{ll}
        \frac12 L_{B[\underline{\tensX},\bar{\varepsilon}(\underline{\tensX})]} & \text{if}\ \tensX=0,\\
        \|\nabla f(\tensX)\|_\frob\sqrt{d} c(\tensX) \bar{s}+\frac{L_{B[\underline{\tensX},\bar{\varepsilon}(\underline{\tensX})]}}{2}(\sqrt{d}+1)^2 & \text{otherwise.}
    \end{array}\right.\]
    If $\tensX=0$, the claim is ready. If $\tensX\neq 0$, we further simplify $\tilde{a}(\tensX,B[\underline{\tensX},\bar{\varepsilon}(\underline{\tensX})])$ and yield
        \begin{align}
            \tilde{a}(\tensX,B[\underline{\tensX},\bar{\varepsilon}(\underline{\tensX})])
            &=\|\nabla f(\tensX)\|_\frob\sqrt{d} (\frac{1}{\sigma_{\min}(\tensX)^2}+1)^{\frac{d}{2}}+\frac{L_{B[\underline{\tensX},\bar{\varepsilon}(\underline{\tensX})]}}{2}(\sqrt{d}+1)^2\nonumber\\
            &\leq\frac{3\sqrt{d}}{2}\|\nabla f(\underline{\tensX})\|_\frob (\frac{1}{\sigma_{\min}(\tensX)^2}+1)^{\frac{d}{2}}+\frac{L_{B[\underline{\tensX},\bar{\varepsilon}(\underline{\tensX})]}}{2}(\sqrt{d}+1)^2\label{eq: aim of proof GRAP-R}
        \end{align}
    Since the continuity of $\tensX\mapsto(\frac{1}{\sigma_{\min}(\tensX)^2}+1)$ and a lower bound of $\|\tensV_\tensX\|_\frob$ depend on the Tucker ranks of $\tensX$ and $\underline{\tensX}$, we discuss three cases: $\ranktc(\underline{\tensX})=\vecr$, $\ranktc(\underline{\tensX})<\vecr$, and otherwise.

    \textbf{Case 1:}
    If $\ranktc(\underline{\tensX})=\vecr$, $\tensV_{\underline{\tensX}}$ is the Riemannian gradient of $f$, which is continuous on $\tensM_\vecr$. Therefore, there exists $\varepsilon_2(\underline{\tensX})>0$, such that 
    \[\frac12\|\tensV_{\underline{\tensX}}\|_\frob\leq\|\tensV_\tensX\|_\frob\leq\frac{3}{2}\|\tensV_{\underline{\tensX}}\|_\frob\quad\text{and}\quad\frac12\sigma_{\min}(\underline{\tensX})\leq\sigma_{\min}(\tensX)\leq\frac{3}{2}\sigma_{\min}(\underline{\tensX}).\]
    Consequently, let $\varepsilon(\underline{\tensX})=\min\{\varepsilon_1(\underline{\tensX}),\varepsilon_2(\underline{\tensX})\}$ and
    \[\delta(\underline{\tensX})=\frac{a\tilde{\omega}^2\|\tensV_{\underline{\tensX}}\|_\frob^2}{4}\min\{\frac{\rho(1-a)}{\bar{a}(\underline{\tensX})},\rho,\frac{2\rho}{3\|\nabla f(\underline{\tensX})\|_\frob}\},\] 
    where $\bar{a}(\underline{\tensX})=\frac{3\sqrt{d}}{2}\|\nabla f(\underline{\tensX})\|_\frob (\frac{4}{\sigma_{\min}(\underline{\tensX})^2}+1)^{\frac{d}{2}}+\frac{L_{B[\underline{\tensX},\bar{\varepsilon}(\underline{\tensX})]}}{2}(\sqrt{d}+1)^2$. We obtain that 
    \[f(\tensY)\leq f(\tensY_{\vecr})\leq f(\tensX)-\delta(\underline{\tensX}).\]

    \textbf{Case 2:}
    If $\ranktc(\underline{\tensX})<\vecr$, let
    \[\delta(\underline{\tensX})=\frac{a\tilde{\omega}^2}{12}\|\nabla f(\underline{\tensX})\|_\frob\prod_{k=1}^d\frac{r_k-\underline{r}_k}{\min\{n_k,n_{-k}\}}\min\{\frac{\rho(1-a)}{\bar{a}(\underline{\tensX})},\rho,\frac{2\rho}{3\|\nabla f(\underline{\tensX})\|_\frob}\},\]
    where 
    \begin{equation} 
        \label{eq: bar a GRAP}
        \bar{a}(\underline{\tensX})=
        \left\{
            \begin{array}{ll}
                \frac12 L_{B[\underline{\tensX},\bar{\varepsilon}(\underline{\tensX})]} & \text{if}\ \underline{\tensX}=0,\\
                \frac{3\sqrt{d}}{2}\|\nabla f(\underline{\tensX})\|_\frob (\frac{4}{\sigma_{\min}(\underline{\tensX})^2}+1)^{\frac{d}{2}}+\frac{L_{B[\underline{\tensX},\bar{\varepsilon}(\underline{\tensX})]}}{2}(\sqrt{d}+1)^2 & \text{otherwise.}
            \end{array}
        \right.
    \end{equation}
    Note that $\delta(\underline{\tensX})$ only depends on $\underline{\tensX}$. 
    
    Given $\delta(\underline{\tensX})$, we aim to construct $\varepsilon(\underline{\tensX})$. Since $f$ is continuous at $\underline{\tensX}$, there exists $\varepsilon_0(\underline{\tensX})>0$, such that $f(B[\underline{\tensX},\varepsilon_0(\underline{\tensX})])\subseteq[f(\underline{\tensX})-\delta(\underline{\tensX}),f(\underline{\tensX})+\delta(\underline{\tensX})]$. We define
    \[\varepsilon(\underline{\tensX})=\left\{\begin{array}{ll}
        \min\{\frac{\Delta}{2},\frac{1}{\sqrt{d}+1}\varepsilon_0(\underline{\tensX}),\frac{1}{\sqrt{d}+1}\varepsilon_1(\underline{\tensX})\} & \text{if}\ \underline{\tensX}=0,\\
        \min\{\frac{\Delta}{2},\frac{1}{\sqrt{d}+1}\varepsilon_0(\underline{\tensX}),\frac{1}{\sqrt{d}+1}\varepsilon_1(\underline{\tensX}),\frac{1}{2(\sqrt{d}+1)}\sigma_{\min}(\underline{\tensX})\} & \text{otherwise.}
    \end{array}\right.\]
    Denote $\underline{\vecr}=\ranktc(\underline{\tensX})$. For $\tensX\in B[\underline{\tensX},\varepsilon(\underline{\tensX})]\cap\tensM_{\leq\vecr}$, it follows from~\cref{prop: rank delta} that $\rank_\Delta(\matx_{(k)})\leq\underline{r}_k\leq\rank(\matx_{(k)})$ and thus $\tensX_{\underline{\vecr}}:=\proj^{\ho}_{\underline{\vecr}}(\tensX)$ exists. It follows from~\cref{lem: HOSVD inequality} that
    \[\|\tensX_{\underline{\vecr}}-\underline{\tensX}\|_\frob\leq (\sqrt{d}+1)\|\tensX-\underline{\tensX}\|_\frob\leq\min\{\varepsilon_0(\underline{\tensX}),\varepsilon_1(\underline{\tensX}),\frac{\sigma_{\min}(\underline{\tensX})}{2}\}.\]
    Therefore, we obtain that $f(\tensX_{\underline{\vecr}})\leq f(\tensX) + 2\delta(\underline{\tensX})$.
    
    Subsequently, it suffices to provide an upper bound of $\tilde{a}(\tensX,B[\underline{\tensX},\bar{\varepsilon}(\underline{\tensX})])$ in~\cref{eq: aim of proof GRAP-R} and a lower bound of~$\|\tensV_{\tensX_{\underline{\vecr}}}\|_\frob$. If $\underline{\tensX}\neq 0$, for the upper bound of~\cref{eq: aim of proof GRAP-R}, we observe that $\sigma_{\min}(\tensX_{\underline{\vecr}})\in[\frac12\sigma_{\min}(\underline{\tensX}),\frac{3}{2}\sigma_{\min}(\underline{\tensX})]$, and thus
    \[\tilde{a}(\tensX_{\underline{\vecr}},B[\underline{\tensX},\bar{\varepsilon}(\underline{\tensX})])\leq\frac{3\sqrt{d}}{2}\|\nabla f(\underline{\tensX})\|_\frob (\frac{4}{\sigma_{\min}(\underline{\tensX})^2}+1)^{\frac{d}{2}}+\frac{L_{B[\underline{\tensX},\bar{\varepsilon}(\underline{\tensX})]}}{2}(\sqrt{d}+1)^2,\]
    which only depends on $\underline{\tensX}$. For the lower bound of $\|\tensV_{\tensX_{\underline{\vecr}}}\|_\frob$, it follows from~\cref{eq: inequality} that
    \[\|\tensV_{\tensX_{\underline{\vecr}}}\|_\frob\geq \|\nabla f(\tensX_{\underline{\vecr}})\|_\frob\prod_{k=1}^d\sqrt{\frac{r_k-\underline{r}_k}{\min\{n_k,n_{-k}\}}}\geq\frac{1}{2}\|\nabla f(\underline{\tensX})\|_\frob\prod_{k=1}^d\sqrt{\frac{r_k-\underline{r}_k}{\min\{n_k,n_{-k}\}}}.\]
    Let $\tensY_{\underline{\vecr}}$ generated by GRAP for one step with initial guess $\tensX_{\underline{\vecr}}$. It holds that 
    \begin{align*}
        f(\tensY)&\leq f(\tensY_{\underline{\vecr}})\leq f(\tensX_{\underline{\vecr}})-a\tilde{\omega}^2\|\tensV_{\tensX_{\underline{\vecr}}}\|_\frob^2\min\{\frac{\rho(1-a)}{\tilde{a}(\tensX_{\underline{\vecr}},B[\underline{\tensX},\bar{\varepsilon}(\underline{\tensX})])},\rho,\frac{2\rho}{3\|\nabla f(\underline{\tensX})\|_\frob}\}\\
        &\leq f(\tensX_{\underline{\vecr}})-\frac{a\tilde{\omega}^2}{4}\|\nabla f(\underline{\tensX})\|_\frob\prod_{k=1}^d\frac{r_k-\underline{r}_k}{\min\{n_k,n_{-k}\}}\min\{\frac{\rho(1-a)}{\bar{a}(\underline{\tensX})},\rho,\frac{2\rho}{3\|\nabla f(\underline{\tensX})\|_\frob}\}\\
        &\leq f(\tensX) + 2\delta(\underline{\tensX}) - 3\delta(\underline{\tensX})\\
        &\leq f(\tensX) - \delta(\underline{\tensX}).
    \end{align*}
    Note that if $\underline{\tensX}=0$, we obtain that $\tensX_{\underline{\vecr}}=0$ and prove the same result with the different $\bar{a}(\underline{\tensX})=L_{B[\underline{\tensX},\bar{\varepsilon}(\underline{\tensX})]}/2$.
    
    \textbf{Case 3:}
    If $\underline{\tensX}\in\tensM_{\leq\vecr}\setminus(\tensM_{\vecr}\cup\tensM_{<\vecr})$, we denote $I=\{k\in[d]:\underline{r}_k<r_k\}$ and consider the Tucker decomposition $\underline{\tensX}=\underline{\tensG}\times_{k=1}^d\underline{\matu}_k$ of $\underline{\tensX}$. The proof sketch is similar to the case 2. Let 
    \[\delta(\underline{\tensX})=\frac{a\tilde{\omega}^2}{3}\bar{\delta}(\underline{\tensX})\min\{\frac{\rho(1-a)}{\bar{a}(\underline{\tensX})},\rho,\frac{2\rho}{3\|\nabla f(\underline{\tensX})\|_\frob}\},\]
    where $\bar{a}(\underline{\tensX})$ is given in~\cref{eq: bar a GRAP} and
    \[\bar{\delta}(\underline{\tensX})=\frac12\min\{c_I\|\nabla f(\underline{\tensX})\times_{k\in[d]\setminus I}\proj_{\underline{\matu}_k}\|_\frob^2,\sum_{k=1}^{d}\|(\nabla f(\underline{\tensX}))_{(k)}\proj_{\Span(\underline{\matx}_{(k)}^\top)}\|_\frob^2\}\] 
    with $c_I=\prod_{k\in I}(r_k-\underline{r}_k)/\min\{n_k,n_{-k}\}$.
    Since $f$ is continuous, there exists $\varepsilon_0(\underline{\tensX})\in (0,\varepsilon_3(\underline{\tensX}))$, such that $f(B[\underline{\tensX},\varepsilon_0(\underline{\tensX})])\subseteq[f(\underline{\tensX})-\delta(\underline{\tensX}),f(\underline{\tensX})+\delta(\underline{\tensX})]$, where we defer the definition of $\varepsilon_3(\underline{\tensX})$. Let
    \[\varepsilon(\underline{\tensX})=\min\{\frac{\Delta}{2},\frac{1}{\sqrt{d}+1}\varepsilon_0(\underline{\tensX}),\frac{1}{\sqrt{d}+1}\varepsilon_1(\underline{\tensX}),\frac{1}{2(\sqrt{d}+1)}\sigma_{\min}(\underline{\tensX})\}.\] 
    For $\tensX\in B[\underline{\tensX},\varepsilon(\underline{\tensX})]\cap\tensM_{\leq\vecr}$, it follows from~\cref{prop: rank delta} that $\rank_\Delta(\matx_{(k)})\leq\underline{r}_k\leq\rank(\matx_{(k)})$ for $k\in I$ and thus $\tensX_{\underline{\vecr}}=\proj^{\ho}_{\underline{\vecr}}(\tensX)$ exists. We can obtain
    \[\|\tensX_{\underline{\vecr}}-\underline{\tensX}\|_\frob\leq\varepsilon_0(\underline{\tensX})\leq\varepsilon_3(\underline{\tensX})\quad\text{and}\quad f(\tensX_{\underline{\vecr}})\leq f(\tensX) + 2\delta(\underline{\tensX})\]
    in a same fashion as the case 2. 
    
    Subsequently, it suffices to provide an upper bound of~\cref{eq: aim of proof GRAP-R} and a lower bound of $\|\tensV_{\tensX_{\underline{\vecr}}}\|_\frob$. We can obtain the upper bound of~\cref{eq: aim of proof GRAP-R}
    \[\tilde{a}(\tensX_{\underline{\vecr}},B[\underline{\tensX},\bar{\varepsilon}(\underline{\tensX})])\leq\frac{3\sqrt{d}}{2}\|\nabla f(\underline{\tensX})\|_\frob (\frac{4}{\sigma_{\min}(\underline{\tensX})^2}+1)^{\frac{d}{2}}+\frac{L_{B[\underline{\tensX},\varepsilon_1(\underline{\tensX})]}}{2}(\sqrt{d}+1)^2\]
    in a same fashion as the case 2.

    For the lower bound of $\|\tensV_{\tensX_{\underline{\vecr}}}\|_\frob$, the proof is different from the case 2 since $\tensX_{\underline{\vecr}}\in\tensM_{\leq\vecr}\setminus(\tensM_{\vecr}\cup\tensM_{<\vecr})$. Recall the projection $\tensV_{\tensX_{\underline{\vecr}}}=\proj_{\tangent_{\tensX_{\underline{\vecr}}}\!\tensM_{\leq\vecr}}(-\nabla f(\tensX_{\underline{\vecr}}))$ in~\cref{eq: exact projection} and the Tucker decomposition $\tensX_{\underline{\vecr}}=\hat{\tensG}\times_{k=1}^d\hat{\matu}_k$ of $\tensX_{\underline{\vecr}}$. We observe that 
        \begin{align*}
            \|\tensV_{\tensX_{\underline{\vecr}}}\|_\frob^2
            &=\|\nabla f(\tensX_{\underline{\vecr}})\times_{k\in I}\proj_{[\hat{\matu}_k\ \hat{\matu}_{k,1}^*]}\times_{k\in[d]\setminus I}\proj_{\hat{\matu}_k}\|_\frob^2\\
            &~~~+ \sum_{k\in I}\|\proj^\perp_{[\hat{\matu}_k\ \hat{\matu}_{k,1}^*]}(\nabla f(\tensX_{\underline{\vecr}}))_{(k)}\hat{\matv}_k\hat{\matG}_{(k)}^\dagger\hat{\matG}_{(k)}^{}\hat{\matv}_k^\top\|_\frob^2 \\
            &~~~+ \sum_{k\in[d]\setminus I}\|\proj^\perp_{\hat{\matu}_k}(\nabla f(\tensX_{\underline{\vecr}}))_{(k)}\hat{\matv}_k\hat{\matG}_{(k)}^\dagger\hat{\matG}_{(k)}^{}\hat{\matv}_k^\top\|_\frob^2,
        \end{align*}
    where $\hat{\matv}_k=(\hat{\matu}_j)^{\otimes j\neq k}$ and $\hat{\matG}_{(k)}^\dagger=\hat{\matG}_{(k)}^\top(\hat{\matG}_{(k)}^{}\hat{\matG}_{(k)}^\top)^{-1}$. Since $\|\tensV_{\tensX_{\underline{\vecr}}}\|_\frob\neq 0$ (or the algorithm terminates otherwise), there exists (at least) one non-negative term. If $\|\nabla f(\tensX_{\underline{\vecr}})\times_{k\in I}\proj_{[\hat{\matu}_k\ \hat{\matu}_{k,1}^*]}\times_{k\in[d]\setminus I}\proj_{\hat{\matu}_k}\|_\frob\neq 0$, it follows from~\cref{eq: inequality} that 
        \begin{align*}
            \|\tensV_{\tensX_{\underline{\vecr}}}\|_\frob^2&\geq\|\nabla f(\tensX_{\underline{\vecr}})\times_{k\in[d]\setminus I}\proj_{\hat{\matu}_k}\|_\frob^2\prod_{k\in I}\frac{r_k-\underline{r}_k}{\min\{n_k,n_{-k}\}}>0.
        \end{align*}
    Since $\proj_{\hat{\matu}_k}=\proj_{\Span{(\hat{\matx}_{(k)})}}$ does not depend on a specific Tucker decomposition of~$\tensX_{\underline{\vecr}}$, there exists $\varepsilon_3(\underline{\tensX})>0$, such that 
    \[\|\tensV_{\tensX_{\underline{\vecr}}}\|_\frob^2\geq c_I\|\nabla f(\tensX_{\underline{\vecr}})\times_{k\in[d]\setminus I}\proj_{\hat{\matu}_k}\|_\frob^2\geq\frac{c_I}{2}\|\nabla f(\underline{\tensX})\times_{k\in[d]\setminus I}\proj_{\underline{\matu}_k}\|_\frob^2\]
    for all $\tensX_{\underline{\vecr}}\in B[\underline{\tensX},\varepsilon_3(\underline{\tensX})]\cap\tensM_{\underline{\vecr}}$. If $\|\nabla f(\tensX_{\underline{\vecr}})\times_{k\in I}\proj_{[\hat{\matu}_k\ \hat{\matu}_{k,1}^*]}\times_{k\in[d]\setminus I}\proj_{\hat{\matu}_k}\|_\frob=0$, it holds that $\nabla f(\tensX_{\underline{\vecr}})\times_{k\in I}\proj_{[\hat{\matu}_k\ \hat{\matu}_{k,1}^*]}\times_{k\in[d]\setminus I}\hat{\matu}_k^\top=0$. Therefore, we obtain that
        \begin{align*}
            \|\tensV_{\tensX_{\underline{\vecr}}}\|_\frob^2&=\sum_{k=1}^{d}\|(\nabla f(\tensX_{\underline{\vecr}}))_{(k)}\hat{\matv}_k^{}\hat{\matG}_{(k)}^\dagger\hat{\matG}_{(k)}^{}\hat{\matv}_k^{\top}\|_\frob^2>0.
        \end{align*}
    Since $\hat{\matv}_k^{}\hat{\matG}_{(k)}^\dagger\hat{\matG}_{(k)}^{}\hat{\matv}_k^{\top}=\proj_{\Span(\hat{\matv}_k^{}\hat{\matG}_{(k)}^\top)}=\proj_{\Span(\hat{\matx}_{(k)}^\top)}$ does not depend on a specific Tucker decomposition, there exists $\varepsilon_3(\underline{\tensX})>0$, such that 
    \[\frac{1}{2}\sum_{k=1}^{d}\|(\nabla f(\underline{\tensX}))_{(k)}\proj_{\Span(\underline{\matx}_{(k)}^\top)}\|_\frob^2\leq\|\tensV_{\tensX_{\underline{\vecr}}}\|_\frob^2\leq\frac{3}{2}\sum_{k=1}^{d}\|(\nabla f(\underline{\tensX}))_{(k)}\proj_{\Span(\underline{\matx}_{(k)}^\top)}\|_\frob^2\]
    for all $\tensX_{\underline{\vecr}}\in B[\underline{\tensX},\varepsilon_3(\underline{\tensX})]\cap\tensM_{\underline{\vecr}}$. Subsequently, $\|\tensV_{\tensX_{\underline{\vecr}}}\|_\frob$ can be bounded below from $\bar{\delta}(\underline{\tensX})$, i.e.,
        \begin{align*}
            \|\tensV_{\tensX_{\underline{\vecr}}}\|_\frob^2&\geq\frac12\min\{c_I\|\nabla f(\underline{\tensX})\times_{k\in[d]\setminus I}\proj_{\underline{\matu}_k}\|_\frob^2,\sum_{k=1}^{d}\|(\nabla f(\underline{\tensX}))_{(k)}\proj_{\Span(\underline{\matx}_{(k)}^\top)}\|_\frob^2\}.
        \end{align*}

    Consequently, let $\tensY_{\underline{\vecr}}$ be generated by GRAP for one step with initial guess $\tensX_{\underline{\vecr}}$. It holds that 
        \begin{align*}
            f(\tensY)&\leq f(\tensY_{\underline{\vecr}})\leq f(\tensX_{\underline{\vecr}})-a\tilde{\omega}^2\|\tensV_{\tensX_{\underline{\vecr}}}\|_\frob^2\min\{\frac{\rho(1-a)}{\tilde{a}(\tensX_{\underline{\vecr}},B[\underline{\tensX},\bar{\varepsilon}(\underline{\tensX})])},\rho,\frac{2\rho}{3\|\nabla f(\underline{\tensX})\|_\frob}\}\\
            &\leq f(\tensX_{\underline{\vecr}})-a\tilde{\omega}^2\bar{\delta}(\underline{\tensX})\min\{\frac{\rho(1-a)}{\bar{a}(\underline{\tensX})},\rho,\frac{2\rho}{3\|\nabla f(\underline{\tensX})\|_\frob}\}\\
            &\leq f(\tensX) + 2\delta(\underline{\tensX}) - 3\delta(\underline{\tensX})\\
            &\leq f(\tensX) - \delta(\underline{\tensX}).
        \end{align*}
\end{proof}

It is worth noting that if $\ranktc(\tensX)<\vecr$, it follows from~\cref{lem: inequality,prop: angle condition GRAP} that $\|\nabla f(\tensX)\|_\frob$ can be upper-bounded by the norm of the approximate projection $\|\tilde{\tensV}_\tensX\|_\frob$, rendering a different analysis in the proof (Cases~2 and~3). The proofs of Cases~1 and~2 are inspired from~\cite{olikier2026low}. However, tensors in $\tensM_{\leq\vecr}\setminus(\tensM_{\vecr}\cup\tensM_{<\vecr})$ lead to a different proof, where the sufficient decrease condition is built via the analysis of summands of the approximate projection $\tilde{\proj}$ and continuity of subspaces $\Span(\matx_{(k)})$ and $\Span(\matx_{(k)}^\top)$ with respect to $\tensX$. In other words, tensors where some but not all modes are of full rank indeed render essential difficulties in developing provable first-order methods. By using~\cref{prop: GRAP-R decrease}, we can prove that GRAP-R method is apocalypse-free.
\begin{theorem}\label{thm: GRAP-R}
    Given $f:\tensM_{\leq\vecr}\to\mathbb{R}$ bounded below from $f^*$ with locally Lipschitz continuous gradient and compact sublevel set. Let $\{\tensX^{(t)}\}$ be a sequence generated by~\cref{alg: GRAP-R}. If $\{\tensX^{(t)}\}$ is finite, the~\cref{alg: GRAP-R} terminates at a stationary point. Otherwise, $\{\tensX^{(t)}\}$ has an accumulation point and any accumulation point of $\{\tensX^{(t)}\}$ is stationary.
\end{theorem}
\begin{proof}
    We prove the result by contradiction. Suppose that there is a subsequence $\{\tensX^{(t_j)}\}_{j=0}^{\infty}$ converging to a non-stationary point $\underline{\tensX}$. There exists $T>0$, such that $\|\tensX^{(t_j)}-\underline{\tensX}\|_\frob\leq\varepsilon(\underline{\tensX})$ for $j\geq T$. Therefore, it follows from~\cref{prop: GRAP-R decrease} that
        \begin{align*}
            f^*-f(\tensX^{(t_T)})&\leq f(\underline{\tensX})-f(\tensX^{(t_T)})=\sum_{j=T}^{\infty} (f(\tensX^{(t_{j+1})}) - f(\tensX^{(t_{j})}))\\
            &\leq\sum_{j=T}^{\infty} (f(\tensX^{(t_{j}+1)}) - f(\tensX^{(t_{j})}))\leq\sum_{t=T+1}^{\infty}-\delta(\underline{\tensX})=-\infty,
        \end{align*}
    which contradicts with $f^*>-\infty$.
\end{proof}

\section{Retraction-free GRAP-R method}
\label{sec: rfGRAP-R}
Based on the retraction-free partial projection~\cref{eq: retraction-free projection}, we develop the retraction-free gradient-related approximate projection method with rank decrease (rfGRAP-R). Since the partial projection~\cref{eq: retraction-free projection} is retraction-free, we propose a different rank-decreasing procedure that only monitors the last singular value of unfolding matrices. We prove that the rfGRAP-R method converges to stationary points in a similar fashion as the GRAP-R method.

\subsection{Algorithm}
The rfGRAP-R method is illustrated in~\cref{alg: rfGRAP-R}. Specifically, given $\tensX^{(t)}\in\tensM_{\leq\vecr}$, we only detect the smallest non-zero singular value of unfolding matrices $\matx_{(1)}^{(t)},\matx_{(2)}^{(t)},\dots,\matx_{(d)}^{(t)}$, and construct the index sets 
\begin{equation}
    \label{eq: rfGRAP index}
    I_k=\left\{\begin{array}{ll}
            \{\rank(\matx_{(k)}^{(t)})-1,\rank(\matx_{(k)}^{(t)})\}& \text{if}\ \sigma_{\min}(\matx_{(k)}^{(t)})\leq\Delta,\\
            \{\rank(\matx_{(k)}^{(t)})\} & \text{otherwise.}
        \end{array}\right.
\end{equation}
Then, we project $\tensX^{(t)}$ to $\tensM_{\breve{\vecr}}$ for all $\breve{\vecr}\in I_1\times I_2\times\cdots\times I_d$ and yield lower-rank candidates $\tensX_{\breve{\vecr}}^{(t)}=\proj_{\breve{\vecr}}^{\mathrm{HO}}(\tensX^{(t)})$. 

Subsequently, we perform line searches to each lower-rank candidate. Given $\tensX_{\breve{\vecr}}^{(t)}$, the line search along the partial projection $\hat{\proj}_{\tangent_{\tensX_{\breve{\vecr}}^{(t)}}\!\tensM_{\leq\vecr}}(-\nabla f(\tensX_{\breve{\vecr}}^{(t)}))$ in~\cref{eq: an approximate projection} is given by
\begin{equation}
    \label{eq: rfGRAP mapping}
    \tensY_{\breve{\vecr}}^{(t)}=\tensX_{\breve{\vecr}}^{(t)}+s_{\breve{\vecr}}^{(t)}\hat{\proj}_{\tangent_{\tensX_{\breve{\vecr}}^{(t)}}\!\tensM_{\leq\vecr}}(-\nabla f(\tensX_{\breve{\vecr}}^{(t)}))
\end{equation}
For the selection of stepsize, we consider the Armijo backtracking line search in~\cref{eq: Armijo stepsize}. Finally, the $\tensX^{(t+1)}$ is set to be the point in $\{\tensY^{(t)}_{\breve{\vecr}}\}_{\breve{\vecr}}$ that has the lowest function value. The rfGRAP-R method terminates if the approximate projection satisfies $\|\hat{\proj}_{\tangent_{\tensX^{(t)}}\!\tensM_{\leq\vecr}}(-\nabla f(\tensX^{(t)}))\|_\frob=0$, i.e., $\tensX^{(t)}$ is a stationary point of~\cref{eq: problem (P)} due to the angle condition~\cref{eq: angle condition rfGRAP}.

\begin{algorithm}[htbp]
    \caption{Retraction-free gradient-related approximate projection method with rank decrease (rfGRAP-R)}
    \label{alg: rfGRAP-R}
    \begin{algorithmic}[1]
        \REQUIRE Initial guess $\tensX^{(0)}\in\tensM_{\leq\vecr}$, backtracking parameters $\rho, a\in(0,1)$, $t=0$.        
        \WHILE{The stopping criteria are not satisfied}
            \STATE Compute the index sets 
                $I_k$ for $k\in[d]$ via~\cref{eq: rfGRAP index}.
            \FOR{$\breve{\vecr}\in I_1\times I_2\times\cdots\times I_d$}
                \STATE Compute lower-rank candidate $\proj_{\breve{\vecr}}^{\mathrm{HO}}(\tensX^{(t)})$ by~\cref{eq: HOSVD}.
                \STATE Compute $\tensY_{\breve{\vecr}}^{(t)}$ by~\cref{eq: rfGRAP mapping} with Armijo backtracking line search~\cref{eq: Armijo stepsize}.
            \ENDFOR
            \STATE Update $\tensX^{(t+1)}=\argmin_{\breve{\vecr}\in I_1\times I_2\times\cdots\times I_d}f(\tensY_{\breve{\vecr}}^{(t)})$ and $t=t+1$.\label{line: rfGRAP-R}
        \ENDWHILE
        \ENSURE $\tensX^{(t)}$.
    \end{algorithmic}
\end{algorithm}

\begin{remark}
    Both~\cref{alg: GRAP-R,alg: rfGRAP-R} consist of a rank-decreasing mechanism and line search; as in~\cref{fig: general illustration}. However, in contrast with the GRAP-R method (\cref{alg: GRAP-R}) where $\Delta$-ranks are monitored and $(\prod_{k=1}^d r_k^{(t)}+1)$ candidates are generated in the worst case, the proposed {rfGRAP-R} method in~\cref{alg: rfGRAP-R} only monitors one singular value $\sigma_{\min}(\matx_{(k)}^{(t)})$ of unfolding matrices, resulting in a much smaller index set $I_1\times I_2\times\cdots\times I_d$ with $2^d$ elements in the worst case, i.e., a smaller number of rank-decreasing attempts. The rationale behind this is that the partial projection $\hat{\proj}$ is retraction-free, rendering a different analysis in the stepsize of backtracking line search; see~\cref{prop: backtracking rfGRAP,prop: rfGRAP}. In practice, if the rank parameter is small, the GRAP-R method performs better, since it benefits from taking full advantage of the tangent cone. If the rank is over-estimated or the rank parameter becomes large, the rfGRAP-R method performs better by avoiding exhaustive exploration of multiple lower-rank candidates; see~\cref{fig: biased}. Additionally, one may consider different $\Delta_k>0$ for each mode to construct the index sets $I_k$ in~\cref{eq: GRAP-R index,eq: rfGRAP index} in practice. We set $\Delta=\Delta_1=\Delta_2=\cdots=\Delta_d$ for the sake of brevity.
\end{remark}

\begin{remark}
    The proposed methods are closely-related to the rank reduction methods for bounded-rank matrices in~\cite{olikier2023apocalypse,olikier2026low}. However, there are two essential differences. First, as mentioned above, developing provable methods for low-rank tensors presents intrinsic difficulties in projections and convergence analysis. Second, there is a fundamental difference in the rank-reduction mechanism of the retraction-free methods between the rfGRAP-R and the ``RFDR" method in~\cite{olikier2023apocalypse}. Specifically, rfGRAP-R monitors the smallest nonzero singular value of an iterate, whereas RFDR tracks the $r$-th singular value along each mode, restricting the search of points in~$\mathbb{R}^{m\times n}_{r-1} \cup \mathbb{R}^{m\times n}_{r}$.
\end{remark}

\subsection{Backtracking line search}
We develop a lower bound of the stepsize in~\cref{eq: Armijo stepsize} for partial projection~\cref{eq: retraction-free projection} by using the Lipschitz constant $L_{B}$ in~\cref{eq: Lipschitz}. For the sake of brevity, we denote $\tensV_\tensX=\proj_{\tangent_{\tensX}\!\tensM_{\leq\vecr}}(-\nabla f(\tensX))$ and $\hat{\tensV}_\tensX=\hat{\proj}_{\tangent_{\tensX}\!\tensM_{\leq\vecr}}(-\nabla f(\tensX))$.
\begin{proposition}\label{prop: backtracking rfGRAP}
    Given $\tensX\in\tensM_{\leq\vecr}$ satisfying $\|\tensV_\tensX\|_\frob\neq 0$ and $\bar{s}>0$. Consider the closed ball $B_\tensX=B[\tensX,\bar{s}\|\tensV_\tensX\|_\frob]$. The Armijo condition
    \begin{equation}
        \label{eq: Armijo rfGRAP}
        f(\tensX)-f(\tensX+s\hat{\tensV}_\tensX)\geq s a \hat{\omega}(\tensX)^2\|\tensV_\tensX\|_\frob^2
    \end{equation}
    holds for some $s\in[\underline{s}(B_\tensX),\bar{s}]$ with $\underline{s}(B_\tensX)=\min\{\rho(1-a)/\hat{a}(B_\tensX),\rho\bar{s}\}$, where $\hat{a}(B_\tensX)=L_{B_\tensX}/2$. 
\end{proposition}
\begin{proof}
    Denote $\tensY=\tensX+s\hat{\tensV}_\tensX$. It follows from~\cref{eq: Lipschitz} and~\cref{prop: approximate rfGRAP} that
        \begin{align*}
            f(\tensY)-f(\tensX)&\leq\langle\nabla f(\tensX),\tensY-\tensX\rangle+\frac{L_{B_\tensX}}{2}\|\tensY-\tensX\|_\frob^2=s\|\hat{\tensV}_\tensX\|_\frob^2(-1+s\frac{L_{B_\tensX}}{2}).\\
            &\leq s\hat{\omega}(\tensX)^2\|\tensV_\tensX\|_\frob^2(-1+s\hat{a}(\tensX,B_\tensX))
        \end{align*}
    for all $s\in[0,\bar{s}]$.

    Moreover, we observe that $s\|\tensV_\tensX\|_\frob^2(-1+s\hat{a}(B_\tensX))\leq -s a \|\tensV_\tensX\|_\frob^2$ if and only if $s\leq(1-a)/\hat{a}(B_\tensX)$. Therefore, \cref{eq: Armijo stepsize} can be satisfied for all $s\in[\rho(1-a)/\hat{a}(B_\tensX),(1-a)/\hat{a}(B_\tensX)]$. If $\bar{s}\leq(1-a)/\hat{a}(B_\tensX)$, the Armijo condition~\cref{eq: Armijo stepsize} holds for some $s\in[\rho\bar{s},\bar{s}]$.
\end{proof}

It is worth noting that in contrast with $\tilde{a}(B[\tensX,\sqrt{d}\bar{s}\|\tensV_\tensX\|_\frob])$ in~\cref{prop: backtracking GRAP}, the constant $\hat{a}(B[\tensX,\bar{s}\|\tensV_\tensX\|_\frob])$ in~\cref{prop: backtracking rfGRAP} does not depend on the singular values of $\tensX$ since the rfGRAP-R method searches along a straight line without retraction in~\cref{eq: Armijo rfGRAP}. On the contrary, the backtracking linesearch~\cref{eq: Armijo GRAP} in the GRAP-R method searches along a ``curve''.

\subsection{Convergence to stationary points}
Similar to the GRAP-R method, we prove that the rfGRAP-R method converges to stationary points via the sufficient decrease condition in~\cref{subsec: GRAP-R convergence}.

\begin{proposition}\label{prop: rfGRAP}
    For all $\underline{\tensX}\in\tensM_{\vecr}$ and $\|\tensV_{\underline{\tensX}}\|_\frob\neq 0$, there exists $\varepsilon(\underline{\tensX})>0$ and $\delta(\underline{\tensX})>0$, such that $f(\tensY)-f(\tensX)<\delta(\underline{\tensX})$ holds for all $\tensX\in B[\underline{\tensX},\varepsilon(\underline{\tensX})]\cap\tensM_{\leq\vecr}$ and~$\tensY$ generated by the rfGRAP-R method (\cref{alg: rfGRAP-R}) for one step with initial guess $\tensX$.
\end{proposition}
\begin{proof}
    The proof sketch of~\cref{prop: rfGRAP} is illustrated in~\cref{fig: proof sketch rfGRAP-R}, which is similar to~\cref{prop: GRAP-R decrease}. Specifically, given $\underline{\tensX}\in\tensM_{\leq\vecr}$ with $\underline{\vecr}=\ranktc(\underline{\tensX})$. If $\underline{\vecr}=\vecr$, the result can be proved by the continuity of the projected anti-gradient. Otherwise, we project $\tensX\in\varepsilon(\underline{\tensX})$ onto $\tensM_{\breve{\vecr}}$ for some $\breve{\vecr}$, and adopt line search in~\cref{eq: rfGRAP mapping} to yield $\tensY_{\breve{\vecr}}$. Consequently, we obtain that 
    \[f(\tensY)\leq f(\tensY_{\breve{\vecr}})\leq f(\proj^{\mathrm{HO}}_{\breve{\vecr}}(\tensX))-3\delta(\underline{\tensX})\leq f(\tensX)-\delta(\underline{\tensX})\]
    by continuity of $f$. 
    \begin{figure}[htbp]
    
        \centering
        
        \begin{tikzpicture}
            \def\h{2};
            \node[draw,rounded corners] (O) at (0,0) {Goal: $f(\tensY)-f(\tensX)\leq-\delta(\underline{\tensX})<0$, for all $\|\tensX-\underline{\tensX}\|_\frob\leq\varepsilon(\underline{\tensX})$};

            \node[draw,rounded corners,inner sep=10pt] (T) at ($(O)-(0,2)$) {\begin{minipage}{10cm}
                \centering\vspace{-2mm}
                Three inequalities\vspace{5mm}\\
                $f(\tensY)\quad\leq\quad f(\tensY_{\breve{\vecr}})\quad\leq\quad f(\proj^{\mathrm{HO}}_{\breve{\vecr}}(\tensX))-3\delta(\underline{\tensX})\quad\leq\quad f(\tensX)-\delta(\underline{\tensX})$
            \end{minipage}};

            \draw[->,arrows = {-Latex[width=5pt, length=5pt]}] (T) -- (O);

            \coordinate (C) at ($(T)-(3.6,0.3)$);
            \def\x{1.45};
            \def\y{0.25};
            \draw[densely dotted] ($(C)+(\x,\y)$) -- ($(C)+(\x,-\y)$) -- ($(C)+(-\x,-\y)$) -- ($(C)+(-\x,\y)$) -- cycle;
            
            \node[draw,densely dotted] (F) at ($(T)+(-4.5,-\h)$) {Line~\ref{line: rfGRAP-R} in~\cref{alg: rfGRAP-R}};
            \draw[->,arrows = {-Latex[width=5pt, length=5pt]}] (F) -- ($(C)-(0,\y)$);

            \coordinate (C) at ($(T)-(0.5,0.3)$);
            \def\x{2.6};
            \def\y{0.5};
            \draw[dashed,rounded corners] ($(C)+(\x,\y)$) -- ($(C)+(\x,-\y)$) -- ($(C)+(-\x,-\y)$) -- ($(C)+(-\x,\y)$) -- cycle;

            \node[draw,dashed,rounded corners] (S) at ($(T)+(0,-\h)$) {Armijo condition~\cref{eq: Armijo rfGRAP}};
            \draw[->,arrows = {-Latex[width=5pt, length=5pt]}] (S) -- ($(T)-(0,\y)-(0,0.3)$);

            \coordinate (C) at ($(T)+(1.95,-0.3)$);
            \def\x{3.1};
            \def\y{0.32};
            \draw[densely dotted] ($(C)+(\x,\y)$) -- ($(C)+(\x,-\y)$) -- ($(C)+(-\x,-\y)$) -- ($(C)+(-\x,\y)$) -- cycle;

            \node[draw,densely dotted] (Third) at ($(T)+(4.5,-\h)$) {\cref{lem: HOSVD inequality}};
            \draw[->,arrows = {-Latex[width=5pt, length=5pt]}] (Third) -- ($(C)+(0.5*\x,-\y)$);
            \draw[->,arrows = {-Latex[width=5pt, length=5pt]}] (Third) -- (S);

            \node[draw,densely dotted] (Lem1) at ($(S)+(0,-0.7*\h)$) {Angle condition~\cref{eq: angle condition rfGRAP}};
            \draw[->,arrows = {-Latex[width=5pt, length=5pt]}] (Lem1) -- (S);

            \node[draw,densely dotted] (Lem3) at ($(Third)+(0,-0.7*\h)$) {Quasi optimality~\cref{eq: quasi-optimal HOSVD}};
            \draw[->,arrows = {-Latex[width=5pt, length=5pt]}] (Lem3) -- (Third);
        \end{tikzpicture}
        \caption{Proof sketch of~\cref{prop: rfGRAP}.}
        \label{fig: proof sketch rfGRAP-R}
    \end{figure}
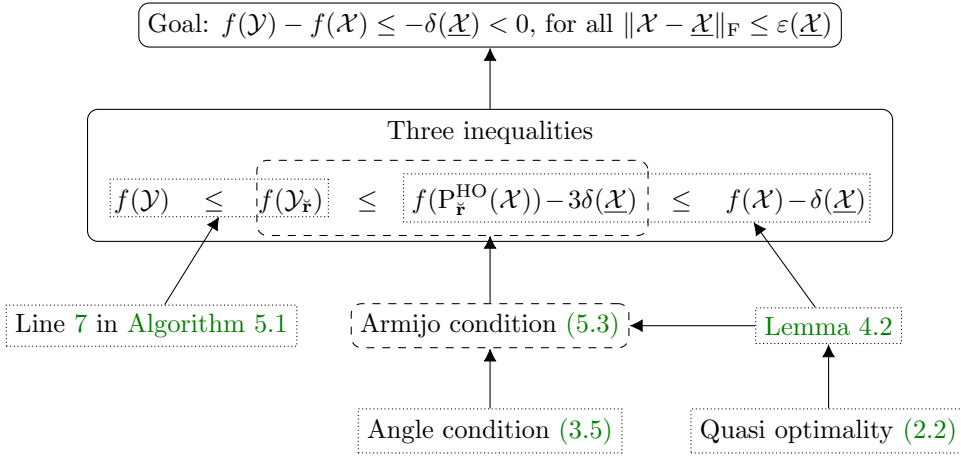
    
    First, we recall the variables $\varepsilon_1(\underline{\tensX})$ and $\bar{\varepsilon}(\underline{\tensX})=3\sqrt{d+1}\bar{s}\|\nabla f(\underline{\tensX})\|_\frob/2+\varepsilon_1(\underline{\tensX})$ and $B[\tensX,\sqrt{d+1}\bar{s}\|\tensV_\tensX\|_\frob]\subseteq B[\underline{\tensX},\bar{\varepsilon}(\underline{\tensX})]$ in the proof of~\cref{prop: GRAP-R decrease}. Subsequently, we can adopt~\cref{eq: Armijo rfGRAP} to $\tensX\in B[\underline{\tensX},\varepsilon_1(\underline{\tensX})]\cap\tensM_{\leq\vecr}$ and closed ball $B[\tensX,\bar{s}\|\tensV_\tensX\|_\frob]$, and we obtain that
        \begin{align*}
            f(\tensX)-f(\tensX+\underline{s}(B[\tensX,\bar{s}\|\tensV_\tensX\|_\frob])\hat{\tensV}_\tensX)&\geq a \underline{s}(B[\tensX,\bar{s}\|\tensV_\tensX\|_\frob])  \hat{\omega}(\tensX)^2\|\tensV_\tensX\|_\frob^2\\
            &\geq a \underline{s}(B[\underline{\tensX},\bar{\varepsilon}(\underline{\tensX})]) \hat{\omega}(\tensX)^2\|\tensV_\tensX\|_\frob^2\\
            &\geq \frac{a \underline{s}(B[\underline{\tensX},\bar{\varepsilon}(\underline{\tensX})])}{d+1}\prod_{k=1}^{d}\frac{1}{\min\{n_k,n_{-k}\}}\|\tensV_\tensX\|_\frob^2,
        \end{align*}
    where we use the fact $\hat{\omega}(\tensX)^2\geq\frac{1}{d+1}\prod_{k=1}^{d}\frac{1}{\min\{n_k,n_{-k}\}}$. Then, it suffices to bound $\|\tensV_\tensX\|_\frob$ below from terms by $\underline{\tensX}$, which depends on $\underline{\vecr}=\ranktc(\underline{\tensX})$.

    \textbf{Case 1:}
    If $\ranktc(\underline{\tensX})=\vecr$, $\tensV_{\underline{\tensX}}$ boils down to the Riemannian gradient of $f$, which is continuous on $\tensM_\vecr$. Therefore, there exists $\varepsilon_2(\underline{\tensX})>0$, such that 
    \[\frac12\|\tensV_{\underline{\tensX}}\|_\frob\leq\|\tensV_\tensX\|_\frob\leq\frac{3}{2}\|\tensV_{\underline{\tensX}}\|_\frob.\]
    Consequently, let
    \[\varepsilon(\underline{\tensX})=\min\{\varepsilon_1(\underline{\tensX}),\varepsilon_2(\underline{\tensX})\}\quad\text{and}\quad\delta(\underline{\tensX})=\frac{a \underline{s}(B[\underline{\tensX},\bar{\varepsilon}(\underline{\tensX})])}{4(d+1)}\prod_{k=1}^{d}\frac{1}{\min\{n_k,n_{-k}\}}\|\tensV_{\underline{\tensX}}\|_\frob^2.\]
    We have 
    \[f(\tensY)\leq f(\tensY_{\vecr})\leq f(\tensX)-\delta(\underline{\tensX}).\]

    \textbf{Case 2:}
    If $\ranktc(\underline{\tensX})\neq\vecr$, let
    \[\delta(\underline{\tensX})=\frac{a \underline{s}(B[\underline{\tensX},\bar{\varepsilon}(\underline{\tensX})])}{12(d+1)}\prod_{k=1}^{d}\frac{c_I}{\min\{n_k,n_{-k}\}^2}\|\tensV_{\underline{\tensX}}\|_\frob^2,\]
    where $I=\{k\in[d]:\underline{r}_k<r_k\}$. Since $f$ is continuous at $\underline{\tensX}$, there exists $\varepsilon_0(\underline{\tensX})\in(0,\varepsilon_3(\underline{\tensX}))$, such that $f(B[\underline{\tensX},\varepsilon_0(\underline{\tensX})])\subseteq[f(\underline{\tensX})-\delta(\underline{\tensX}),f(\underline{\tensX})+\delta(\underline{\tensX})]$. We define
    \[\varepsilon(\underline{\tensX})=\min\{\Delta,\frac{1}{\sqrt{d}+1}\varepsilon_0(\underline{\tensX}),\frac{1}{\sqrt{d}+1}\varepsilon_1(\underline{\tensX})\}.\]

    For $\tensX\in B[\underline{\tensX},\varepsilon(\underline{\tensX})]\cap\tensM_{\leq\vecr}$, it follows from~\cref{prop: rank delta} that $\rank_\Delta(\matx_{(k)})\leq\underline{r}_k\leq\rank(\matx_{(k)})$ for all $k\in I$ and $\rank(\matx_{(k)})=r_k=\underline{r}_k$ for $k\in[d]\setminus I$. Consider $\breve{\vecr}=(\breve{r}_1,\breve{r}_2,\dots,\breve{r}_d)$ with $\breve{r}_k=\min\{r_k-1,\rank(\matx_{(k)})\}$ for $k\in I$ and $\breve{r}_k=r_k$ for $k\in[d]\setminus I$. Since $\breve{\vecr}\leq\ranktc(\tensX)$, $\tensX_{\breve{\vecr}}=\proj^{\ho}_{\breve{r}}(\tensX)$ exists. It follows from~\cref{lem: HOSVD inequality} that
    \[\|\tensX_{\breve{\vecr}}-\underline{\tensX}\|_\frob\leq(\sqrt{d}+1)\|\tensX-\underline{\tensX}\|_\frob\leq\min\{\varepsilon_0(\underline{\tensX}),\varepsilon_1(\underline{\tensX})\}.\]
    Therefore, $f(\tensX_{\breve{\vecr}})\leq f(\tensX) + 2\delta(\underline{\tensX})$.

    We aim to bound $\|\tensV_{\tensX_{\breve{\vecr}}}\|_\frob$ by the terms depending on $\underline{\tensX}$. We observe from~\cref{eq: inequality} and retraction-free directions~\cref{eq: retraction-free projection} that there exists $\varepsilon_3(\underline{\tensX})>0$ such that
        \begin{align*}
            \|\tensV_{\tensX_{\breve{\vecr}}}\|_\frob&\geq\sqrt{c_I}\|\nabla f(\tensX_{\breve{\vecr}})\times_{k\in[d]\setminus I}\proj_{\breve{\matu}_k}\|_\frob
            \geq\frac{\sqrt{c_I}}{2}\|\nabla f(\underline{\tensX})\times_{k\in[d]\setminus I}\proj_{\underline{\matu}_k}\|_\frob,
        \end{align*}
    and
        \begin{align*}
            \|\tensV_{\tensX_{\breve{\vecr}}}\|_\frob&\geq\|\tilde{\proj}_{k}(\nabla f(\tensX_{\breve{\vecr}}))\|_\frob=\|\breve{\tensG}\times_k(\proj_{\breve{\matu}_{k}}^\perp(\nabla f(\tensX_{\breve{\vecr}})\times_{j\neq k}\breve{\matu}_j^\T)_{(k)}\breve{\matG}_{(k)}^\dagger)\times_{j\neq k}\breve{\matu}_j\|_\frob\\
            &\geq\frac12\|\underline{\tensG}\times_k(\proj_{\underline{\matu}_{k}}^\perp(\nabla f(\underline{\tensX})\times_{j\neq k}\underline{\matu}_j^\T)_{(k)}\underline{\matG}_{(k)}^\dagger)\times_{j\neq k}\underline{\matu}_j\|_\frob
        \end{align*}
    for all $k\in[d]$ and $\tensX_{\breve{\vecr}}\in B[\underline{\tensX},\varepsilon_3(\underline{\tensX})]$. Similar to case 3 in the proof of~\cref{prop: GRAP-R decrease}, the analysis above does not depend on a specific choice of Tucker decompositions $\tensX_{\breve{\vecr}}=\breve{\tensG}\times_{k=1}^d\breve{\matu}_k$ and $\underline{\tensX}=\underline{\tensG}\times_{k=1}^d\underline{\matu}_k$. Consequently, we have
    \[\|\tensV_{\tensX_{\breve{\vecr}}}\|_\frob\geq\frac{\sqrt{c_I}}{2}\prod_{k=1}^d\sqrt{\frac{1}{\min\{n_k,n_{-k}\}}}\|\hat{\tensV}_{\underline{\tensX}}\|_\frob\geq\frac{\sqrt{c_I}}{2}\prod_{k=1}^d\frac{1}{\min\{n_k,n_{-k}\}}\|\tensV_{\underline{\tensX}}\|_\frob\]

    Finally, we bound $f(\tensY)$ via backtracking line search on $\tensX_{\breve{\vecr}}$. Since 
    \[\|\tensZ-\underline{\tensX}\|_\frob\leq\|\tensZ-\tensX_{\breve{\vecr}}\|_\frob+\|\tensX_{\breve{\vecr}}-\underline{\tensX}\|_\frob\leq \bar{s}\|\tensV_{\tensX_{\breve{\vecr}}}\|_\frob + (\sqrt{d}+1)\|\tensX-\underline{\tensX}\|_\frob\leq\bar{\varepsilon}(\underline{\tensX})\]   
    for all $\tensZ\in B[\tensX_{\breve{\vecr}},\bar{s}\|\tensV_{\tensX_{\breve{\vecr}}}\|_\frob]$, it holds that $B[\tensX_{\breve{\vecr}},\bar{s}\|\tensV_{\tensX_{\breve{\vecr}}}\|_\frob]\subseteq B[\underline{\tensX},\bar{\varepsilon}(\underline{\tensX})]$. Consequently, we obtain from~\cref{eq: Armijo rfGRAP} and $\|\tensX_{\breve{\vecr}}-\underline{\tensX}\|_\frob\leq\varepsilon_1(\underline{\tensX})$ that 
    \begin{align*}
        f(\tensY)&\leq f(\tensY_{\breve{\vecr}})\leq f(\tensX_{\breve{\vecr}}) - \frac{a \underline{s}(B[\underline{\tensX},\bar{\varepsilon}(\underline{\tensX})])}{d+1}\prod_{k=1}^{d}\frac{1}{\min\{n_k,n_{-k}\}}\|\tensV_{\tensX_{\breve{\vecr}}}\|_\frob^2\\
        &\leq f(\tensX_{\breve{\vecr}}) - \frac{a \underline{s}(B[\underline{\tensX},\bar{\varepsilon}(\underline{\tensX})])}{4(d+1)}\prod_{k=1}^{d}\frac{c_I}{\min\{n_k,n_{-k}\}^2}\|\tensV_{\underline{\tensX}}\|_\frob^2\\
        &\leq f(\tensX) + 2\delta(\underline{\tensX}) - 3\delta(\underline{\tensX})\\
        &\leq f(\tensX) - \delta(\underline{\tensX}).
    \end{align*}
\end{proof}

It is worth noting that Cases 2 and 3 in the proof of~\cref{prop: GRAP-R decrease} are treated separately, while they are merged in the proof of~\cref{prop: rfGRAP}. The rationale is whether the analysis depends on the smallest singular value $\sigma_{\min}(\tensX)$. Specifically, in the proof of~\cref{prop: GRAP-R decrease}, the inequality~\cref{eq: aim of proof GRAP-R} involves the term $1/\sigma_{\min}(\tensX)^2 + 1$. As a result, different rank-deficiency regimes lead to different properties, necessitating a case-by-case analysis. On the contrary, the analysis of the rfGRAP-R method does not depend on $\sigma_{\min}(\tensX)$, and thus, Cases 2 and 3 can be merged into a unified argument. By using~\cref{prop: rfGRAP}, we can prove that the rfGRAP-R method is apocalypse-free in a same fashion as~\cref{thm: GRAP-R}.
\begin{theorem}
    Given $f:\tensM_{\leq\vecr}\to\mathbb{R}$ bounded below from $f^*$ with locally Lipschitz continuous gradient and compact sublevel set. Let $\{\tensX^{(t)}\}$ be a sequence generated by~\cref{alg: rfGRAP-R}. If $\{\tensX^{(t)}\}$ is finite, the~\cref{alg: rfGRAP-R} terminates at a stationary point. Otherwise, $\{\tensX^{(t)}\}$ has an accumulation point and any accumulation point of $\{\tensX^{(t)}\}$ is stationary.
\end{theorem}

\section{Numerical validation}\label{sec: experiments}
We validate the effectiveness of the GRAP-R and rfGRAP-R methods on the tensor completion problem based on Tucker decomposition. Specifically, given a partially observed tensor $\tensA\in\mathbb{R}^{n_1\times n_2\cdots\times n_d}$ on an index set $\Omega\subseteq[n_1]\times [n_2]\times\cdots\times[n_d]$, Tucker tensor completion aims to recover the tensor $\tensA$ from its entries on $\Omega$ based on the low-rank Tucker decomposition. The optimization problem can be formulated on the set of bounded-rank tensors~$\tensM_{\leq\vecr}$, i.e., 
\begin{equation}\label{eq: LRTC P}
    \begin{aligned}
        \min_\tensX\ \ & \frac12\|\proj_\Omega(\tensX)-\proj_\Omega(\tensA)\|_\frob^2\\
        \subjectto\ \ & \quad \quad \tensX\in\tensM_{\leq\vecr},
    \end{aligned}
\end{equation}
where $\proj_\Omega$ is the projection operator onto $\Omega$, i.e, $\proj_\Omega(\tensX)(i_1,\dots,i_d)=\tensX(i_1,\dots,i_d)$ if~$(i_1,\dots,i_d)\in\Omega$, otherwise $\proj_\Omega(\tensX)(i_1,\dots,i_d)=0$ for $\tensX\in\mathbb{R}^{n_1\times n_2\times\cdots\times n_d}$. The \emph{sampling rate} is denoted by $p:=|\Omega|/(n_1n_2\cdots n_d)$. 

We compare the proposed methods with other methods: 1) a Riemannian conjugate gradient method\footnote{GeomCG toolbox: \url{https://www.epfl.ch/labs/anchp/index-html/software/geomcg/}.} (GeomCG)~\cite{kressner2014low}; 2) a Riemannian conjugate gradient method\footnote{Available at: \url{https://bamdevmishra.in/codes/tensorcompletion/}.} on quotient manifold under a preconditioned metric (RCG-quotient)~\cite{kasai2016low} for optimization on fixed-rank manifolds; 3) the gradient-related approximate projection method (GRAP); 4) the retraction-free gradient-related approximate projection method (rfGRAP); and 5) the Tucker rank-adaptive method\footnote{GRAP, rfGRAP, TRAM methods: \url{https://github.com/JimmyPeng1998/TRAM}.} (TRAM) for optimization on Tucker tensor varieties~\cite{gao2025low}. Note that the GRAP and rfGRAP methods correspond to line search methods on bounded-rank Tucker tensors without any rank-decreasing mechanism in the GRAP-R and rfGRAP-R methods. In contrast, the GRAP-R and rfGRAP-R methods involve rank-decreasing mechanisms, which allow the methods to search over multiple lower-rank candidates at each iteration.

The performance of all methods is evaluated by the test error $\varepsilon_{\Gamma}(\tensX):=\|\proj_\Gamma(\tensX)-\proj_\Gamma(\tensA)\|_\frob/\|\proj_\Gamma(\tensA)\|_\frob$, where $\Gamma$ is a test set different from the training set $\Omega$. To ensure a fair comparison, we adopt the default settings in~\cite[\S 7.1]{gao2025low}. We set $d=3$ and $n_1=n_2=n_3=400$. All experiments are performed on a workstation with two Intel(R) Xeon(R) Processors Gold 6330 (at 2.00GHz$\times$28, 42M Cache) and 512GB of RAM running Matlab R2019b under Ubuntu 22.04.3. The codes of the proposed methods are available at \href{https://github.com/JimmyPeng1998/TuckerApoFree}{https://github.com/JimmyPeng1998/TuckerApoFree.}

\subsection{Test with true rank}
First, we examine the performance of all methods with true rank, i.e., $\vecr=\vecr^*=(6,6,6)$. \cref{fig: unbiased}~reports the test error of all methods with sampling rate $p=0.005,0.01,0.05$. First, we observe that the GRAP-R method is comparable to other candidates. Second, the rfGRAP-R method requires more iterations than the GRAP-R method, since it only adopt partial information from the tangent cone to avoid retraction. 

\begin{figure}[htbp]
    \centering
    \includegraphics[width=\textwidth]{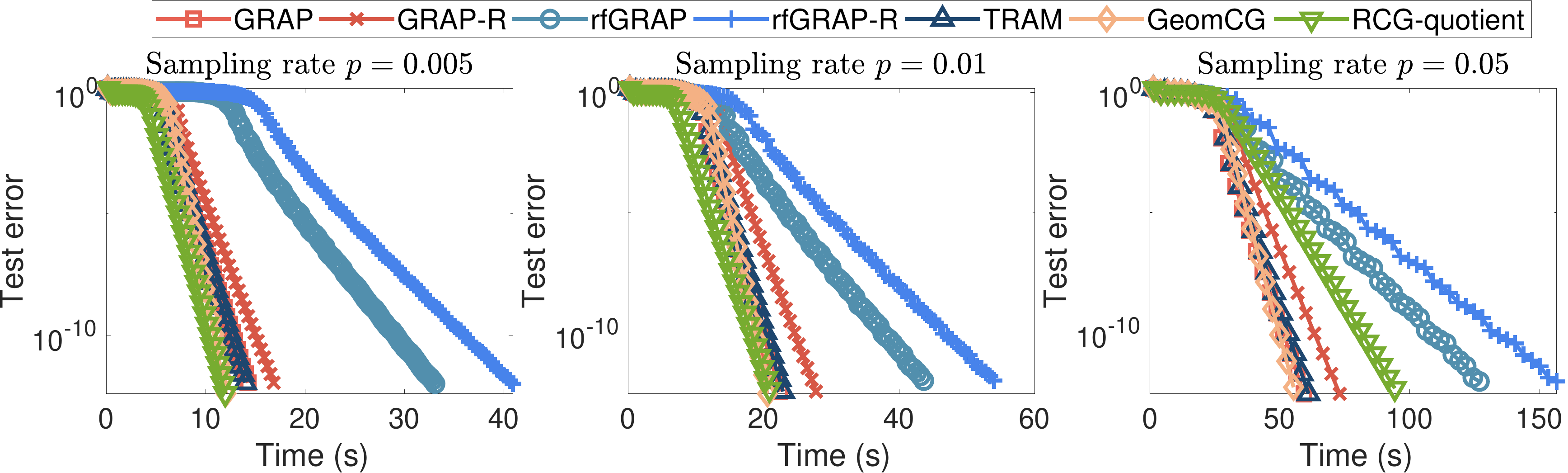}
    \caption{The recovery performance under sampling rate $p=0.005,0.01,0.05$ with unbiased rank parameter.}
    \label{fig: unbiased}
\end{figure}

\subsection{Test with over-estimated rank parameter}
We test the performance of all methods under a set of over-estimated ranks $\vecr=(r,r,r)$ with $r=3,4,5,6>r^*=2$. Note that the constructed rank-deficient stationary point $\tensA$ has the rank $\vecr^*=(r^*,r^*,r^*)=(2,2,2)$. We set the sampling rate $p=0.01$. To ensure a fair comparison to fixed-rank solvers GeomCG and RCG-quotient, the initial guess $\tensX^{(0)}$ is generated from $\tensM_{\vecr}$. The numerical results are reported in~\cref{fig: biased}. First, we observe from~\cref{fig: biased} (left) that the proposed GRAP-R, rfGRAP-R methods, and TRAM method converge while other methods fail to recover the data tensor due to the over-estimated rank parameter. Second, the rfGRAP-R method is comparable to the TRAM method and performs better than the GRAP-R method since the rfGRAP-R method requires fewer rank-decreasing attempts than the GRAP-R method. Third, the right figure in~\cref{fig: biased} reports the history of selected rank parameters in the GRAP-R, rfGRAP-R and TRAM methods. We observe that all these methods are able to recover the underlying low-rank structure of the tensor $\tensA$ even with over-estimated rank-parameter, i.e., converge to a rank-deficient stationary point.

\begin{figure}[htbp]
    \centering
    \includegraphics[width=\textwidth]{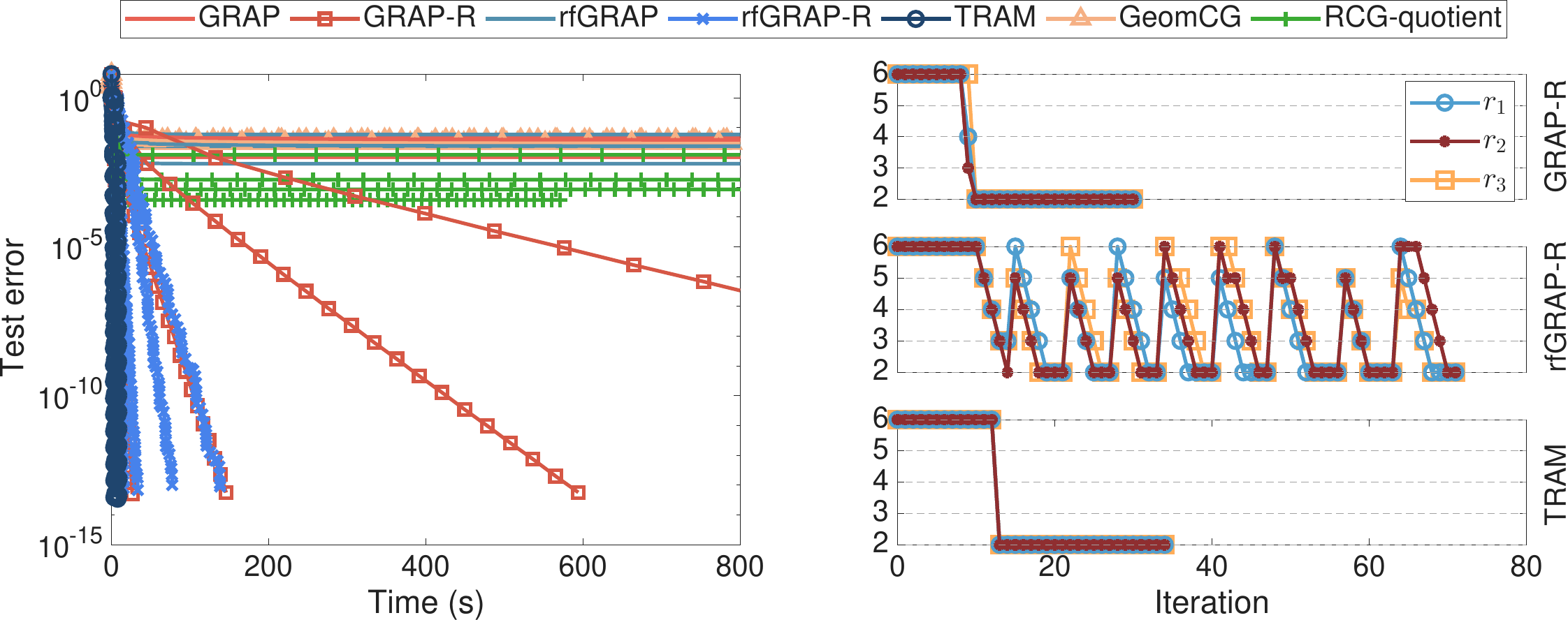}
    \caption{Numerical results with over-estimated rank parameter. Left: test errors. Right: history of the selected rank parameters from the constructed index sets in~\cref{alg: GRAP-R,alg: rfGRAP-R}, and TRAM method with rank parameter $\vecr=(6,6,6)$.}
    \label{fig: biased}
\end{figure}

\subsection{Test on real dataset}
	We also test the proposed two methods (GRAP-R and rfGRAP-R) on the ``Movielens 1M'' dataset\footnote{Available at \url{https://grouplens.org/datasets/movielens/1m/}.}, where the true underlying rank is unknown. These movie ratings can be formulated as a third-order tensor $\tensA$ of size $6040\times 3952\times 150$. We randomly select $80\%$ of the known ratings as a training set $\Omega$ and the rest $20\%$ ratings are test set $\Gamma$. The rank parameter is set to be $\vecr=(r,r,r)$ with $r=1,2,\dots,11$.  We observe from~\cref{fig:ML1M} (left) that the GRAP-R method can converge faster under moderate rank parameter. \Cref{fig:ML1M}~(right) shows that if the rank parameter is relatively small, the GRAP-R performs slightly better than rfGRAP-R. As the rank increases, the performance of rfGRAP-R becomes better than the GRAP-R method. 
	
	\begin{figure}[htbp]
		\centering
		\includegraphics[width=\textwidth]{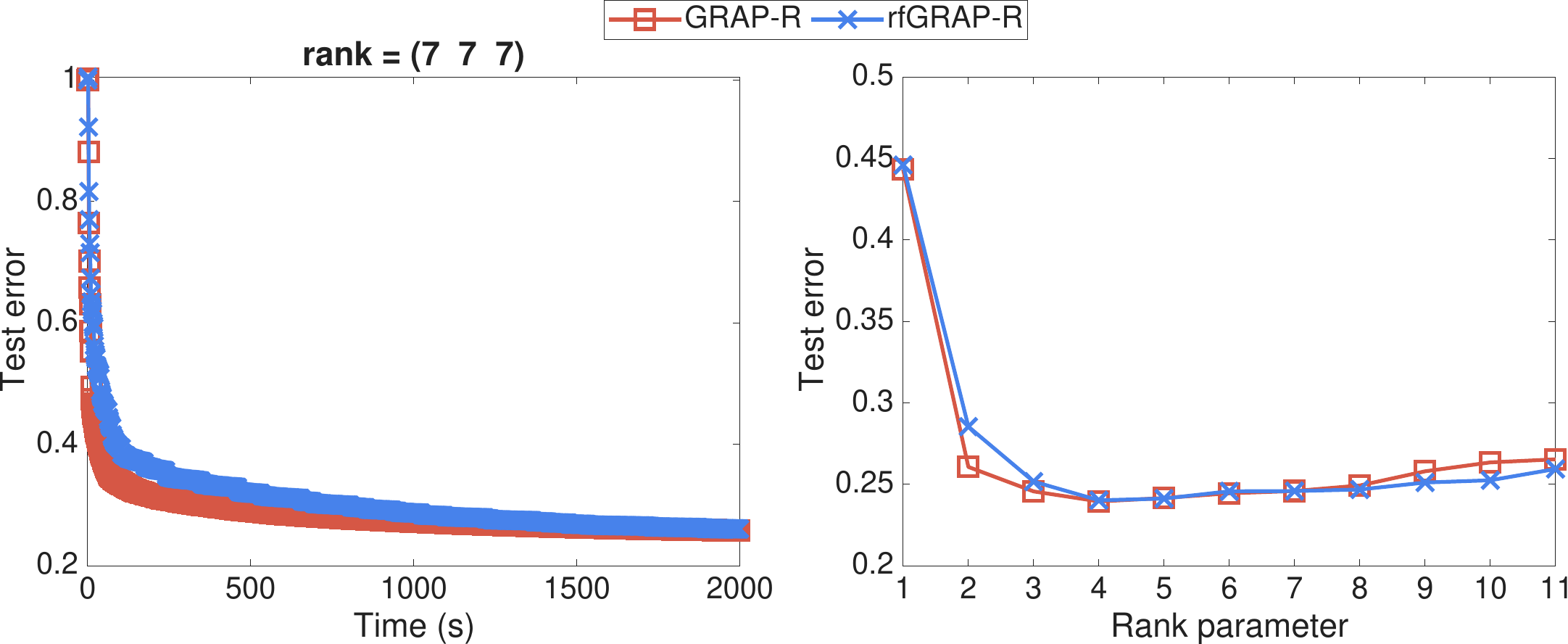}
		\caption{Test errors of the GRAP-R and rfGRAP-R methods on the movielens 1M dataset.}
		\label{fig:ML1M}
	\end{figure}
	
	In summary, GRAP-R is more effective when the prescribed rank is close to the intrinsic rank of the data, while rfGRAP-R is more robust to over-estimated rank parameters.


\section{Conclusion and perspectives}\label{sec: conclusion}
We have proposed two first-order methods, GRAP-R and rfGRAP-R, for low-rank tensor optimization in Tucker format that converge to stationary points, thereby resolving the open problem posed in~\cite[Example 3.14]{levin2023finding}. While developing such methods in matrix case appears to be prosperous recently, provably finding stationary points in the set of bounded-rank tensors turns out to be challenging due to the intricacy of tensors. The first obstacle is a practical optimality condition for low-rank tensor optimization. We have provided an explicit characterization of normal cones, which facilitates a practical optimality condition, and extends the scope of geometry of Tucker tensor varieties in~\cite[\S 3]{gao2025low}. The second obstacle is constructing an ``appropriate'' projection since the metric projection onto the tangent cone does not enjoy a closed form. We have proposed an approximate projection and a retraction-free partial projection via singular value decomposition. The third challenge is that the projection of the anti-gradient onto the tangent cone is not continuous in the set of bounded-rank tensors. As a remedy, we have proposed two gradient-related approximate projection methods by combining rank-decreasing mechanisms and line searches. Numerical results in tensor completion suggest that the GRAP-R and rfGRAP-R methods indeed converge to stationary points under different selections of rank parameters.

The proposed methods borrow the idea of rank reduction on bounded-rank matrices in recent work~\cite{olikier2023apocalypse,olikier2023first,olikier2026low}, but the generalization to tensors has intrinsic difficulties in projections and convergence analysis. It is worth noting that the rfGRAP-R method can be deemed as a new ``rank-adaptive'' method both on bounded-rank matrices ($d=2$) and tensors. On the one hand, the proposed rank-decreasing mechanism enables the precise determination of an appropriate rank parameter. On the other hand, the partial projection~\cref{eq: retraction-free projection} supports ``rank-adaptive'' searching. Specifically, we observe that searching along $\hat{\proj}_0$ can increase the rank of an iterate, and searching along $\proj_k$ preserves the current rank. In summary, both the proposed rank-decreasing mechanism and the partial projection~\cref{eq: retraction-free projection} contribute to identifying the appropriate rank parameter, which has been reflected in numerical experiments.

The proposed methods are guaranteed to converge to stationary points. However, due to the additional computational cost from rank-decreasing mechanisms, the proposed rfGRAP-R method is comparable to the heuristic rank-adaptive method, TRAM~\cite{gao2025low}. Rather than aiming to outperform TRAM in practice, we aim to complement such effective heuristic methods by providing a framework with rigorous theoretical guarantees. In practice, we recommend the rank-adaptive method to avoid line searches at every low-rank candidate. In the light of the observations, a promising research direction is to develop provable methods that avoid (or alleviate) lengthy rank explorations.


\section*{Acknowledgments}
We would like to thank the editor and the anonymous reviewer for insightful comments, P.-A. Absil for comments on the methods and proofs, Nicolas Boumal and Christopher Criscitiello for discussions on normal cones of bounded-rank tensors, and Guillaume Olikier for discussions on stepsizes in line search methods.

\bibliographystyle{siamplain}
\bibliography{references}

\begin{thebibliography}{10}

\bibitem{absil2009optimization}
{\sc P.-A. Absil, R.~Mahony, and R.~Sepulchre}, {\em Optimization algorithms on
  matrix manifolds}, in Optimization Algorithms on Matrix Manifolds, Princeton
  University Press, 2009, \url{https://doi.org/10.1515/9781400830244}.

\bibitem{boumal2023intromanifolds}
{\sc N.~Boumal}, {\em An introduction to optimization on smooth manifolds},
  Cambridge University Press, 2023,
  \url{https://doi.org/10.1017/9781009166164}.

\bibitem{boumal2019global}
{\sc N.~Boumal, P.-A. Absil, and C.~Cartis}, {\em Global rates of convergence
  for nonconvex optimization on manifolds}, IMA Journal of Numerical Analysis,
  39 (2019), pp.~1--33, \url{https://doi.org/10.1093/imanum/drx080}.

\bibitem{de2000multilinear}
{\sc L.~De~Lathauwer, B.~De~Moor, and J.~Vandewalle}, {\em A multilinear
  singular value decomposition}, SIAM journal on Matrix Analysis and
  Applications, 21 (2000), pp.~1253--1278.

\bibitem{de2008tensor}
{\sc V.~De~Silva and L.-H. Lim}, {\em Tensor rank and the ill-posedness of the
  best low-rank approximation problem}, SIAM Journal on Matrix Analysis and
  Applications, 30 (2008), pp.~1084--1127,
  \url{https://doi.org/10.1137/06066518X}.

\bibitem{dong2022new}
{\sc S.~Dong, B.~Gao, Y.~Guan, and F.~Glineur}, {\em New {Riemannian}
  preconditioned algorithms for tensor completion via polyadic decomposition},
  SIAM Journal on Matrix Analysis and Applications, 43 (2022), pp.~840--866,
  \url{https://doi.org/10.1137/21M1394734}.

\bibitem{gao2024desingularization}
{\sc B.~Gao, R.~Peng, and Y.-x. Yuan}, {\em Desingularization of bounded-rank
  tensor sets}, arXiv preprint arXiv:2411.14093,  (2024).

\bibitem{gao2024riemannian}
{\sc B.~Gao, R.~Peng, and Y.-x. Yuan}, {\em Riemannian preconditioned
  algorithms for tensor completion via tensor ring decomposition},
  Computational Optimization and Applications, 88 (2024), pp.~443--468,
  \url{https://doi.org/10.1007/s10589-024-00559-7}.

\bibitem{gao2025low}
{\sc B.~Gao, R.~Peng, and Y.-x. Yuan}, {\em Low-rank optimization on {Tucker}
  tensor varieties}, Mathematical Programming, 214 (2025), pp.~357--407,
  \url{https://doi.org/10.1007/s10107-024-02186-w}.

\bibitem{gao2025optimization}
{\sc B.~Gao, R.~Peng, and Y.-x. Yuan}, {\em Optimization on product manifolds
  under a preconditioned metric}, SIAM Journal on Matrix Analysis and
  Applications, 46 (2025), pp.~1816--1845,
  \url{https://doi.org/10.1137/24M1643773}.

\bibitem{grasedyck2010hierarchical}
{\sc L.~Grasedyck}, {\em Hierarchical singular value decomposition of tensors},
  SIAM journal on matrix analysis and applications, 31 (2010), pp.~2029--2054,
  \url{https://doi.org/10.1137/090764189}.

\bibitem{ha2020equivalence}
{\sc W.~Ha, H.~Liu, and R.~F. Barber}, {\em An equivalence between critical
  points for rank constraints versus low-rank factorizations}, SIAM Journal on
  Optimization, 30 (2020), pp.~2927--2955,
  \url{https://doi.org/10.1137/18M1231675}.

\bibitem{hillar2013most}
{\sc C.~J. Hillar and L.-H. Lim}, {\em Most tensor problems are {NP}-hard},
  Journal of the ACM (JACM), 60 (2013), pp.~1--39.

\bibitem{hosseini2019gradient}
{\sc S.~Hosseini and A.~Uschmajew}, {\em A gradient sampling method on
  algebraic varieties and application to nonsmooth low-rank optimization}, SIAM
  Journal on Optimization, 29 (2019), pp.~2853--2880,
  \url{https://doi.org/10.1137/17M1153571}.

\bibitem{jain2014iterative}
{\sc P.~Jain, A.~Tewari, and P.~Kar}, {\em On iterative hard thresholding
  methods for high-dimensional m-estimation}, in Advances in Neural Information
  Processing Systems, Z.~Ghahramani, M.~Welling, C.~Cortes, N.~Lawrence, and
  K.~Weinberger, eds., vol.~27, Curran Associates, Inc., 2014.

\bibitem{kasai2016low}
{\sc H.~Kasai and B.~Mishra}, {\em Low-rank tensor completion: a {Riemannian}
  manifold preconditioning approach}, in Proceedings of The 33rd International
  Conference on Machine Learning, M.~F. Balcan and K.~Q. Weinberger, eds.,
  vol.~48 of Proceedings of Machine Learning Research, New York, New York, USA,
  6 2016, PMLR, pp.~1012--1021.

\bibitem{koch2010dynamical}
{\sc O.~Koch and C.~Lubich}, {\em Dynamical tensor approximation}, SIAM Journal
  on Matrix Analysis and Applications, 31 (2010), pp.~2360--2375,
  \url{https://doi.org/10.1137/09076578X}.

\bibitem{kolda2009tensor}
{\sc T.~G. Kolda and B.~W. Bader}, {\em Tensor decompositions and
  applications}, SIAM review, 51 (2009), pp.~455--500,
  \url{https://doi.org/10.1137/07070111X}.

\bibitem{kressner2014low}
{\sc D.~Kressner, M.~Steinlechner, and B.~Vandereycken}, {\em Low-rank tensor
  completion by {Riemannian} optimization}, BIT Numerical Mathematics, 54
  (2014), pp.~447--468, \url{https://doi.org/10.1007/s10543-013-0455-z}.

\bibitem{levin2023finding}
{\sc E.~Levin, J.~Kileel, and N.~Boumal}, {\em Finding stationary points on
  bounded-rank matrices: a geometric hurdle and a smooth remedy}, Mathematical
  Programming, 199 (2023), pp.~831--864,
  \url{https://doi.org/10.1007/s10107-022-01851-2}.

\bibitem{levin2024effect}
{\sc E.~Levin, J.~Kileel, and N.~Boumal}, {\em The effect of smooth
  parametrizations on nonconvex optimization landscapes}, Mathematical
  Programming,  (2024), pp.~1--49,
  \url{https://doi.org/10.1007/s10107-024-02058-3}.

\bibitem{olikier2023apocalypse}
{\sc G.~Olikier and P.-A. Absil}, {\em An apocalypse-free first-order low-rank
  optimization algorithm with at most one rank reduction attempt per
  iteration}, SIAM Journal on Matrix Analysis and Applications, 44 (2023),
  pp.~1421--1435, \url{https://doi.org/10.1137/22M1518256}.

\bibitem{olikier2023first}
{\sc G.~Olikier, K.~A. Gallivan, and P.-A. Absil}, {\em First-order
  optimization on stratified sets}, arXiv preprint arXiv:2303.16040,  (2023).

\bibitem{olikier2026low}
{\sc G.~Olikier, K.~A. Gallivan, and P.-A. Absil}, {\em Low-rank optimization
  methods based on projected projected-gradient descent that accumulate at
  bouligand stationary points}, Mathematics of Operations Research,  (2026),
  \url{https://doi.org/10.1287/moor.2024.0582}.

\bibitem{olikier2025projected}
{\sc G.~Olikier and I.~Waldspurger}, {\em Projected gradient descent
  accumulates at bouligand stationary points}, SIAM Journal on Optimization, 35
  (2025), pp.~1004--1029, \url{https://doi.org/10.1137/24M1692782}.

\bibitem{rebjock2024optimization}
{\sc Q.~Rebjock and N.~Boumal}, {\em Optimization over bounded-rank matrices
  through a desingularization enables joint global and local guarantees}, arXiv
  preprint arXiv:2406.14211,  (2024).

\bibitem{schneider2015convergence}
{\sc R.~Schneider and A.~Uschmajew}, {\em Convergence results for projected
  line-search methods on varieties of low-rank matrices via {{\L}ojasiewicz}
  inequality}, SIAM Journal on Optimization, 25 (2015), pp.~622--646,
  \url{https://doi.org/10.1137/140957822}.

\bibitem{shalit2012online}
{\sc U.~Shalit, D.~Weinshall, and G.~Chechik}, {\em Online learning in the
  embedded manifold of low-rank matrices}, Journal of Machine Learning
  Research, 13 (2012).

\bibitem{steinlechner2016riemannian}
{\sc M.~Steinlechner}, {\em Riemannian optimization for high-dimensional tensor
  completion}, SIAM Journal on Scientific Computing, 38 (2016), pp.~S461--S484,
  \url{https://doi.org/10.1137/15M1010506}.

\bibitem{tucker1966some}
{\sc L.~R. Tucker}, {\em Some mathematical notes on three-mode factor
  analysis}, Psychometrika, 31 (1966), pp.~279--311,
  \url{https://doi.org/10.1007/BF02289464}.

\bibitem{tucker1964extension}
{\sc L.~R. Tucker et~al.}, {\em The extension of factor analysis to
  three-dimensional matrices}, Contributions to mathematical psychology, 110119
  (1964).

\bibitem{uschmajew2013geometry}
{\sc A.~Uschmajew and B.~Vandereycken}, {\em The geometry of algorithms using
  hierarchical tensors}, Linear Algebra and its Applications, 439 (2013),
  pp.~133--166.

\bibitem{uschmajew2020geometric}
{\sc A.~Uschmajew and B.~Vandereycken}, {\em Geometric methods on low-rank
  matrix and tensor manifolds}, in Handbook of variational methods for
  nonlinear geometric data, Springer, 2020, pp.~261--313.

\bibitem{vandereycken2013low}
{\sc B.~Vandereycken}, {\em Low-rank matrix completion by {Riemannian}
  optimization}, SIAM Journal on Optimization, 23 (2013), pp.~1214--1236,
  \url{https://doi.org/10.1137/110845768}.

\bibitem{vasilescu2003multilinear}
{\sc M.~A.~O. Vasilescu and D.~Terzopoulos}, {\em Multilinear subspace analysis
  of image ensembles}, in 2003 IEEE Computer Society Conference on Computer
  Vision and Pattern Recognition, 2003. Proceedings., vol.~2, IEEE, 2003,
  pp.~II--93, \url{https://doi.org/10.1109/CVPR.2003.1211457}.

\end{thebibliography}
\end{document}